\documentclass{amsart}
\usepackage{amssymb, epic,eepic,epsfig,amsbsy,amsmath,amscd,url}
\usepackage{color}
\usepackage{verbatim} % To be able to use the  which "outcomments" some part of the file.

\def\be#1\ee{\begin{equation}#1\end{equation}}
\theoremstyle{plain}

\newtheorem{Thm}{Theorem}%[section]

\newtheorem{proposition}{Proposition}[section]
\newtheorem{lemma}[proposition]{Lemma}%[section]
\newtheorem*{lemma*}{Lemma}
\newtheorem*{Theorem*}{Theorem}

\newtheorem{thm}[proposition]{Theorem}%[section]

\theoremstyle{definition}

\newtheorem{remark}{Remark}[section]
\newtheorem*{remark*}{Remark}
\newtheorem*{example*}{Example}

\def\printname#1{
    \if\draft
        \smash{\makebox[0pt]{\hspace{-0.5in}
            \raisebox{8pt}{\tt\tiny #1}}}
    \fi
}

\newlength{\standardunitlength}
\catcode`\@=11
\long\def\@makecaption#1#2{%
     \vskip 10pt

\setbox\@tempboxa\hbox{%\ifvoid\tinybox\else\box\tinybox\fi
       \small\sf{\bfcaptionfont #1. }\ignorespaces #2}%
     \ifdim \wd\@tempboxa >\captionwidth {%
         \rightskip=\@captionmargin\leftskip=\@captionmargin
         \unhbox\@tempboxa\par}%
       \else
         \hbox to\hsize{\hfil\box\@tempboxa\hfil}%
     \fi}
\font\bfcaptionfont=cmssbx10 scaled \magstephalf
\newdimen\@captionmargin\@captionmargin=2\parindent
\newdimen\captionwidth\captionwidth=\hsize
\catcode`\@=12

\newcommand{\im}{\operatorname{Im}}

\newcommand{\cC}{{\mathcal C}}

\newcommand{\Tor}{{\operatorname{\mathfrak {tor}}}}

\newcommand{\beq}{\begin{equation}}
\newcommand{\eeq}{\end{equation}}
\newcommand{\bThm}{\begin{Thm}}
\newcommand{\eThm}{\end{Thm}}

\def\BZ{\mathbb Z}

\def\BQ{\mathbb Q}
\def\BR{\mathbb R}
\def\BC{\mathbb C}

\def\BS{\mathbb S}

\def\cR{\mathfrak R}

\def\cL{\mathcal L}

\def\cO{\mathcal O}

\def\la{\langle}
\def\ra{\rangle}

\def\vol{\operatorname{vol}}
\def\Inv{\operatorname{Inv}}

\def\bm{\mathbf m }

\newcommand{\bproof}{\begin{proof}}
\newcommand{\blemma}{\begin{lemma}}
\newcommand{\eproof}{\end{proof}}
\newcommand{\elemma}{\end{lemma}}

\def\fB{\Theta}

\def\coker{\operatorname{coker}}

\def\im{\operatorname{Im}}

\def\rk{\operatorname{rk}}

\def\Hom{\mathrm{Hom}}

\def\al{\alpha}

\def\bk{\mathbf{k}}

\def\br{\mathrm {br}}

\def\bz{\mathbf z}

\def\RR{\BZ[A]}
\newcommand{\pr}{\operatorname{pr}}
\def\tal{\tilde \al}
\def\BM{\mathbb M}
\def\Fix{\mathrm{Fix} }

\begin{document}

\title[Homology Growth]{Homology torsion growth and Mahler measure}

\author{Thang Le}
\address{Department of Mathematics, Georgia Institute of Technology,
Atlanta, GA 30332--0160, USA }
\email{letu@math.gatech.edu}

\thanks{The author was supported in part by National Science Foundation. \\
{ \em Key words: torsion growth, Mahler measure, Alexander polynomials, algebraic dynamical system, entropy, pseudo-isomorphism.}\\
MSC: 57M10, 57M25, 57Q10, 37B50, 37B10.
}
\begin{abstract}

We prove a conjecture of K. Schmidt in algebraic dynamical system theory on the growth of the number of components of fixed point sets.
We also generalize a result of Silver and Williams on the growth of homology torsions of finite abelian covering of link complements.
In both cases, the growth is expressed by the Mahler measure of the first non-zero Alexander polynomial of the corresponding modules.
We use the notion of pseudo-isomorphism, and also tools from commutative algebra and algebraic
geometry, to reduce the conjectures to the case of torsion modules.
We also describe concrete sequences which give the expected values of the limits in both cases. For this part we utilize a  result of
Bombieri and Zannier (conjectured before by A. Schinzel) and a result of Lawton (conjectured before by D. Boyd).

\end{abstract}

\maketitle

%%%%%%%%%%%%%%%%%%%%%%%%%%%%%%%%%%%%%%%%%%%%%%%%%%%%%%%%%%%%%%%%%%%%%%%%%%%%%%%%%%%%%%%%%%%%%%%%%%%%%%%%%%%%%%%%%%%%%%%%%%%%%%%%%%%%%%%%%%%%%%%%%%%%%%%%%%%%%%%%%%%%%%%%%%%%%%%%%%%%%%%%

\section*{Introduction}

\subsection{A conjecture of K. Schmidt} Suppose $M$ is a finitely generated module over the commutative ring $\cR:=\BZ[t_1^{\pm 1},\dots, t_n^{\pm1}]$. Let $\BS$ be the unit circle in
the complex plane $\BC$.   There is a natural action of $\BZ^n$ on the compact abelian group $\hat M = \Hom(M,\BS)$, the Pontryagin dual of $M$. For details on  dynamical systems
of this type the reader is referred to the remarkable book \cite{Schmidt}.
The entropy of this action, denoted by $h(M)$, can be defined in a standard manner. Lind, Schmidt, and Ward \cite{LSW} (see also \cite{EW})  proved that if $M$ is a {\em torsion module}, then
\be
h(M) = \BM (\Delta_0(M)),
\label{e001}
\ee
where $\Delta_0(M)$ is the 0-th Alexander polynomial of $M$ ( also known as the order of $M$), and $\BM(f)$ is the additive Mahler measure of the polynomial $f$. We will recall the definitions of these notions in
Section \ref{not}.

For a subgroup $\Gamma \subset \BZ^n$ of finite index let $\Fix_\Gamma(\hat M)$ be the set of elements of $\hat M$ fixed by actions of elements of $\Gamma$. Then
$\Fix_\Gamma(\hat M)$ is a compact subgroup of $\hat M$ and has a finite number $P_\Gamma(\hat M)$ of connected components.
The following theorem was conjectured by K. Schmidt \cite{Schmidt}, based on  results in the torsion module case.

\begin{Thm} For any finitely generated $\cR$-module $M$ one has
$$ \limsup_{\la \Gamma \ra \to \infty}\frac{\log P_\Gamma(\hat M)}{|\BZ^n /\Gamma|}= h(\Tor (M)). $$
If $n=1$ then one can replace the $\limsup$ by the ordinary $\lim$.
\label{T1}
\end{Thm}

Here $\Tor(M)$ is the torsion submodule of $M$, and
$$ \la \Gamma \ra =   \min \{ |x|, x \in \Gamma \setminus \{0\} \},$$
where $|x|= \sqrt{\sum_i |x_i|^2}$ for $x=(x_1,\dots,x_n) \in \BZ^n$.

The theorem had been proved for the case when $M$ is a torsion module by Schmidt, see \cite[Theorem 21.1]{Schmidt}, and we will make substantial use of this case.

\subsection{A conjecture of Silver and Williams} Suppose $L$ is an oriented link with $n$ ordered components in an oriented  integral homology 3-sphere $Z$, with the complement $X= Z\setminus L$. There is a natural identification
$H_1(X,\BZ) = \BZ^n$. For a subgroup $\Gamma\subset \BZ^n$ of finite index let $X_\Gamma$ be the corresponding abelian covering of $X$, and $X_\Gamma^\br$ the corresponding
branched covering of $Z$. There are defined the Alexander polynomials $\Delta_i(L) \in \cR=\BZ[t_1^{\pm 1},\dots, t_n^{\pm1}]$, $i=0,1,2,\dots$. We will recall the definition of $\Delta_i(L)$ in
section \ref{S4}.

Let $\Delta(L)=\Delta_j(L)$, where $j$ is the smallest
index such that $\Delta_j(L) \neq 0$. For
 an abelian group $G$,  denote $\Tor_\BZ(G)$ the $\BZ$-torsion subgroup of $G$.

\begin{Thm} Notations as above. One has
\be \limsup_{\la \Gamma \ra \to \infty}\frac{\log \left | \Tor_\BZ\big(H_1(X_\Gamma, \BZ) \big)\right|}{|\BZ^n /\Gamma|} =  \BM (\Delta(L)).
\tag{a}
\label{eT2a}
\ee
\be \limsup_{\la \Gamma \ra \to \infty}\frac{\log |\Tor_\BZ(H_1(X_\Gamma^\br, \BZ))|}{|\BZ^n /\Gamma|} =  \BM (\Delta(L)).
\tag{b}
\label{eT2b}
\ee
If $n=1$ then one can replace the $\limsup$ by the ordinary $\lim$.

\label{T2}
\end{Thm}

For the special case when $\Delta(L)= \Delta_0(L)$, part \eqref{eT2b} was proved by Silver and Williams \cite{SW},
who, based on that result, formulated part (b), with the upper limit replaced by the ordinary limit,  as a conjecture. The proof in \cite{SW}  (for the case $\Delta(L) = \Delta_0(L)$), written for $Z=S^3$ and for branched covering only, can be modified for the case of general homology 3-spheres and non-branched coverings. Hence the real new content of Theorem \ref{T2} is the case when $\Delta(L) \neq \Delta_0(L)$. The proof in \cite{SW}
is based on the torsion module case of Theorem 1. It is not surprising that if one can get Theorem 1, then one can generalize the result  of Silver and Williams to the case $\Delta(L) \neq \Delta_0(L)$.

The investigation of the growth of homology torsions of finite covering of knots has a long history, with an interesting conjecture posed by Gordon \cite{Gordon}. The conjecture was proved by
Riley \cite{Riley}  and   Gonzalez-Acuna and  Short
\cite{GS} using Gelfond-Baker results in number theory. Silver and Williams' result mentioned above and Theorem \ref{T2}  are generalizations of Riley and Gonzalez-Acuna and H. Short from the knot to the link case. The surprising appearance of the Mahler measure can be explained from the perspective of $L^2$-torsion theory \cite{Luck_book}: The $L^2$-torsion of the maximal abelian covering of a link complement, at least when $\Delta_0(L) \neq 0$, is the Mahler measure of $\Delta_0(L)$. Theorem \ref{T2}  more or less says that
 the $L^2$-torsion in this case can be approximated by its finite-dimensional counterparts.

In the non-abelian covering case the $L^2$-torsion is given by the hyperbolic volume. So one should expect the similar limit of the left hand side of Equation (a) would give  the hyperbolic volume of the link complements. 
In \cite{Le_tocome}, we will show that, if $L$ is a non-split link in $S^3$,  then

\be \limsup_{\la \Gamma \ra \to \infty} \frac{\log \left | \Tor_\BZ\big(H_1(X_\Gamma, \BZ) \big)\right|}{|\BZ^n /\Gamma|} \le   \frac{\vol(L)}{6\pi}.
\label{eq.202}
\ee
where $\vol(L)$ is the sum of the hyperbolic volumes of the hyperbolic pieces  in the Jaco-Shalen-Johansson decomposition  of $S^3 \setminus L$. Here $\Gamma$ runs the set of all subgroups of $\pi_1(L)$ of finite index,
and $\la \Gamma \ra$ is the minimal word length of $\Gamma \setminus \{1\}$, measured using a fixed finite generator set of $\pi_1$.
In particular, if $\vol(L)=0$, we have the equality in \eqref{eq.202}. For example, if $L$ is a torus knot, then one has equality in \eqref{eq.202}.
For works in this direction see also \cite{Le_slides,Muller,BV}. It is expected that the non-abelian case is much more complicated than the abelian case.

\subsection{An algebraic version of Theorem \ref{T1}} It is not difficult to reformulate Theorem \ref{T1} entirely in terms of module $M$, without going through the Pontryagin dual $\hat M$. We will show that Theorem \ref{T1} is equivalent to the following.

\begin{Thm}  For any finitely generated $\cR$-module $M$ one has
$$ \limsup_{\la \Gamma \ra \to \infty}\frac{\log | \Tor_\BZ (M\otimes \BZ[\BZ^n/\Gamma])  |}{|\BZ^n /\Gamma|}=  \BM (\Delta(\Tor M)). $$
If $n=1$ then one can replace the $\limsup$ by the ordinary $\lim$.
\label{T3}
\end{Thm}

 % , which is always finite if $G$ is a finitely generated abelian group.

  Theorem \ref{T3} is a special case of the following.

\begin{Thm}  Suppose $\cC$ is a chain complex of finitely generated free $\cR$-modules. Then  for every $ i \ge 0$,

$$ \limsup_{\la \Gamma \ra \to \infty}
\frac{\log | \Tor_\BZ (H_i (\cC \otimes  \BZ[\BZ^n / \Gamma]))
|}{|\BZ^n /\Gamma|}=  \BM (\Delta(H_i(
\cC))). $$
If $n=1$ then one can replace the $\limsup$ by the ordinary $\lim$.
\label{T33}
\end{Thm}

In this paper we  will prove Theorem \ref{T3}, and from there deduce Theorem \ref{T2} and \ref{T33}.

\def\tX{\tilde X}

\subsection{Application: abelian covering of $CW$-complex} Suppose $X$ is a finite CW-complex, equipped with a surjection $\rho : H_1(X,\BZ) \to \BZ^n$. Let $\tilde X$ be the abelian covering of $X$ corresponding $\rho$. The $CW$-structure of $X$ lifts to a $CW$-structure of $\tX$.
The group $\BZ^n$ acts as deck transformations on  the covering $\tX$, making the
cellular complex $\cC(\tX)$ of $\tX$  a free finitely-generated $\cR$-complex.

For every subgroup $\Gamma \in \BZ_n$ of finite index let $\rho_\Gamma: \pi_1(X) \to A_\Gamma= \BZ^n/\Gamma$ be the composition
$\pi_1(X) \to H_1(X,\BZ) \overset{\rho}{\longrightarrow} \BZ^n \to A_\Gamma$, where the first map is the abelianization map. Let $X_\Gamma$ be the finite regular covering corresponding to $\rho_\Gamma$.
Apply Theorem \ref{T33} to the complex $\cC(\tX)$ we get the following.

\begin{Thm} Notations as above. Then
$$ \limsup_{\la \Gamma \ra \to \infty}\frac{|\Tor_\BZ (H_1(X_\Gamma, \BZ))|}{|\BZ^n /\Gamma|} = \BM (\Delta(\Tor_\cR H_i(\tX,\BZ))).$$
 If $\ n=1$ then $\limsup$ can be replaced by the ordinary $\lim$.
 \label{T7}
\end{Thm}

\subsection{Ideas of Proofs} To prove Theorem \ref{T3} (and the equivalent Theorem \ref{T1}) we will reduce it to the case when $M$ is a torsion module, which had been proved in \cite{Schmidt},
and the case when $M$ is torsion-free, i.e. when $\Tor(M)=0$. Although the fact  that $M$ is isomorphic to $\Tor(M) \oplus \big( M/\Tor(M)\big)$ is not true in general, it would hold true if we
replace isomorphism by {\em pseudo-isomorphism}, a notion introduced by Bourbaki \cite{Bourbaki}. The notion of pseudo-isomorphism is important for us, and we will review it
in section \ref{not}. The following will be one of the main technical results used in the proof of Theorem \ref{T3}.

\begin{Thm} Suppose $M_1$ and $M_2$ are two pseudo-isomorphic finitely generated $\cR$-modules. Then $ | \Tor_\BZ (M_1\otimes \BZ[\BZ^n/\Gamma])  |$ and $| \Tor_\BZ (M_2\otimes \BZ[\BZ^n/\Gamma])  |$ have the same growth rate
in the sense that
$$\lim_{\la \Gamma \ra \to \infty}\left ( \frac{\log | \Tor_\BZ (M_1\otimes \BZ[\BZ^n/\Gamma])  |}{|\BZ^n /\Gamma|} - \frac{\log | \Tor_\BZ (M_2\otimes \BZ[\BZ^n/\Gamma])  |}{|\BZ^n /\Gamma|}\right)=0.$$
\label{T4}
\end{Thm}

Note that the limit in Theorem \ref{T4} is the ordinary limit, not the upper limit.

%%, and in Theorem \ref{T4} we show that the limit of the difference between tow terms is 0; we still don't know if the limit of each individual term exists.

In general, the direct calculation of $| \Tor_\BZ (M\otimes \BZ[\BZ^n/\Gamma])  |$ (resp.  the exact value or $|\Tor_\BZ(H_1(X_\Gamma^\br,\BZ))|$) is very difficult, especially in the case when $M$ is not a torsion module (resp. the 0-th Alexander polynomial is 0).  The only  known formula for $\Tor_\BZ(X_\Gamma^\br,\BZ)$, due to Mayberry and Murasugi \cite{MM}, applies only to the case when $H_1((X_\Gamma^\br,\BZ))$ itself is a torsion $\BZ$-group.  When the 0-th Alexander polynomial is 0, which is the case
concerned in this paper, the only known result is that of Hillman and Sakuma \cite{HS} who calculated part of the torsion $\Tor_\BZ(H_1(X_\Gamma^\br,\BZ))$. The other not-yet-calculated part is related to the more difficult theory of modular representations of finite groups.

To circumvent this problem, we use an approximation $\beta(\Gamma)$ of $\BZ[\BZ^n/\Gamma]$, for which the calculation of $| \Tor_\BZ (M_1\otimes \beta)  |$ is easier. Here $\beta(\Gamma)$ depends on $\Gamma$
and other data, and it approximates $\BZ[\BZ^n/\Gamma]$ in the sense that $| \Tor_\BZ (M_1\otimes \BZ[\BZ^n/\Gamma])  |$ and $| \Tor_\BZ (M_1\otimes \beta(\Gamma)  |$ have the same growth rate.
The construction of $\beta(\Gamma)$ is based on the theory of torsion points on algebraic varieties. Needless to say, we have to use tools in commutative and  homological algebra to get the desire estimates.

\subsection{Sequence of converging subgroups} Theorem \ref{T3} guarantees there is a sequence of subgroups $\Gamma_s\subset \BZ^n$ of finite index such that
$$ \lim _{s\to \infty}\frac{\log | \Tor_\BZ (M \otimes \BZ[\BZ^n/\Gamma_s])  |}{|\BZ^n /\Gamma_s|} = \BM (\Delta(\Tor M)).$$
In the case when $M$ is a torsion module, half of the proof of  Theorem \ref{T1} in \cite{Schmidt} is to construct such a sequence. The construction is long and
difficult. In Section \ref{last} (see Theorem \ref{T8}) we give  new sequences $\Gamma_s$ that work for both torsion and non-torsion  modules. The proof is probably simpler, because we are able to use  a  result of Bombieri and Zannier
\cite{Zannier,BMZ} on
irreducibility of lacunary polynomials which was conjectured before by Schinzel, and a result of Lawton on approximation of Mahler measure which was a conjecture of Boyd. The methods and results of Section \ref{last} are independent of the other parts and  give an independent proof of  ``half" of Theorem \ref{T3} (or Theorem \ref{T1}), namely that the left hand side of the identity of Theorem \ref{T3}  is greater than or equal to the right hand side.

While writing this paper I was informed by Raimbault \cite{Ri} that he gets an independent result similar to Theorem \ref{T8}  of Section \ref{last} by modifying the sequences in \cite{Schmidt}.
  % However, his proof works only for $n=1$.

\subsection{Acknowledgements} I would like to thank M. Baker, Hailong Dao,  J. Hillman, K. Murasugi, S. Sakuma,  K. Schmidt, D. Silver, and  S.  Williams  for helpful conversations. I am particularly indebted to U.~Zannier who has patiently explained to me  over the course of  2 years many things from diophantine approximations to his (joint with Bombieri and Masser) remarkable work on anomalous subvariety structures, including  a proof of a conjecture of  Schinzel.

I lectured on  parts of  this work at conferences in Fukuoka (March 2009), Trieste (May 2009), Columbia University (June 2009), Osaka City University (November 2009), and would like to thank the organizers for the chance to give talks there.

The paper grew out of my attempt to prove a topological volume conjecture \cite{Le_slides,Le_tocome}. This was part of a program aiming at understanding the question ``Under what conditions $L^2$-torsions can  be approximated by finite group counter parts?". I was attracted to this  program while trying to develop an approach to attack the volume conjecture in quantum topology, and by the beautiful work of  L\"uck work \cite{Luck} on approximation of $L^2$-Betti numbers.

\subsection{Structure of the paper} Section \ref{not} contains notations, basic facts (with some enhancements) about torsion points on algebraic varieties, pseudo-isomorphism, order of modules,
lattices in Hermitian spaces, the integral group ring of finite abelian groups. It also contains a proof that Theorems \ref{T1} and \ref{T3} are equivalent. In Section \ref{approx}, the main technical section,  we present the construction  the approximation $\beta$ of $\BZ[\BZ^n/\Gamma]$. Section \ref{S3} contains proofs of Theorems \ref{T4},  \ref{T3}, and \ref{T33}. Section \ref{S4} gives a proof of Theorem \ref{T2}.
The last section contains the construction of converging sequences of lattices and Theorem \ref{T8}.

\section{Notations and Preliminaries}
\label{not}

\subsection{Modules over $\cR= \BZ[t^{\pm1}_1, \dots, t^{\pm 1}_n]$}
Fix a free abelian group $\BZ^n$. Let $\cR = \BZ[\BZ^n]$, which we identify with $\BZ[t^{\pm1}_1, \dots, t^{\pm 1}_n]$ by sending $\bk=(k_1,\dots,k_n) \in \BZ^n$ to $t^\bk = \prod_{i=1}^kt_i^{k_i}$. The ring $\cR$ is a unique factorization Noetherian domain.
In this paper $\cR$-modules are  supposed to be  finitely generated, and tensor products are assumed over $\cR$ unless otherwise indicated.

For a module $M$ over a commutative domain $R$,  the torsion submodule  $\Tor_R(M)$ is defined by
$$ \Tor_R(M) = \{ x \in M \mid  a x= 0  \text{ for some } 0\neq a \in R\}.$$

An $R$-module $M$ is a {\em torsion module} if $M = \Tor_R M$. If $\Tor_R M=0$, we call  $M$  {\em torsion-free}. If $R=\cR$ we usually drop the subscript $\cR$ in the $\Tor$ notation.

For a subgroup $\Gamma \subset \BZ^n$ let  $A_\Gamma:= \BZ/\Gamma$
and
$I(\Gamma)$  the ideal of $\BZ[t^{\pm1}_1, \dots, t^{\pm 1}_n]$ generated by $\{ 1-t_1^{k_1}\dots t_n^{k_n}, (k_1,\dots,k_n)  \in \Gamma\}$. Then we have the following exact sequence

$$ 0\to I(\Gamma) \to \cR \overset{\pr} {\longrightarrow} \BZ[A_\Gamma] \to 0.$$
 Hence for every $\cR$-module $M$,
 $$ M/I(\Gamma)\,  M \cong M \otimes_\cR \BZ [A_\Gamma].$$

 Suppose $f,g$ are functions with positive real values on the set of subgroups $\Gamma \subset \BZ^n$ of finite index. We say $ f(\Gamma)$ has {\em negligible growth rate} if
$$ \lim_{\la \Gamma \ra \to \infty} f(\Gamma)^{1/|\BZ^n/\Gamma|} = 1.
$$
We say $f$ and $g$ have {\em the same growth rate}, and write   $f\sim g$ if  $f/g$ has negligible growth rate. Note that we do not require the individual limit
 $\lim_{\la \Gamma \ra \to \infty} f(\Gamma)^{1/|\BZ^n/\Gamma|}$ exists in this case.

%% For a finitely-generated $\cR$-module $M$, let
%% $$ t(M,\Gamma) := |\Tor_\BZ(M \otimes \BZ[A_\Gamma])|.$$
We say two $\cR$-modules $M_1$ and $M_2$ {\em have the same torsion growth}, and write
$ M_1 \sim M_2$,
if $$ |\Tor_\BZ(M_1 \otimes \BZ[A_\Gamma])| \sim |\Tor_\BZ(M_2 \otimes \BZ[A_\Gamma])|.$$

\subsection{Alexander polynomials} All  definitions and facts here are standard and can be found in \cite{Turaev,Hillman}.

Every finitely generated $\cR$-module $M$ has a presentation by an exact sequence
$$ \cR^{m_1} \overset{\partial_1}{\longrightarrow} \cR^{m_0} \to M \to 0,
$$
where $\partial_1$, given by a matrix of size $m_1 \times m_{0}$ with entries in $\cR$,  is  known as a {\em presentation matrix} of $M$. A {\em $k$-minor} of $\partial_1$ is the
determinant of any sub-matrix of size $k\times k$ of $\partial_1$.  For $j \ge 0$,
% the $j$-th Fitting ideal  is the ideal $E_j(M) \subset \cR$ generated by all the $(d_0-j)$-minors of $d_1$.
the $j$-th {\em Alexander polynomial} $\Delta_j(M)$ is the greatest common divisor of all the $(m_0-j)$-minor of $\partial_1$.

It is known that
% $E_j(M)$,
$\Delta_j(M)$ depends only on $M$, but not on any particular
presentation matrix. Each $\Delta_j(M)$ is defined up to units in $\cR$, so identity involving
$\Delta_j(M)$ should be understood ``up to units''.

The $0$-th polynomial $\Delta_0(M)$ is known as the {\em order} of $M$, which is non-zero if and only
$M$ is a torsion module. Besides, $\Delta_j(M)$ divides $\Delta_{j-1}(M)$ for every $j \ge 1$.

The {\em rank} of a module $M$ over $\cR$ is the dimension
of the vector space $M \otimes F(\cR)$ over the fractional field $F(\cR)$ of $\cR$.
If $M$ has rank $r$, then $\Delta_j(M) =0$ if $j < r$, and
$$\Delta_{j-r}(M) =\Delta_j(\Tor M).$$
 For any finitely-generated $\cR$-module of rank $r$, define
$$ \Delta (M) := \Delta_r(M) = \Delta_0(\Tor(M)).$$

In case $M =\cR/I$, where $I=(f_1,\dots,f_l)$ is the ideal generated by $f_1,\dots, f_l$, then $\Delta_0(M)= \gcd(f_1,\dots,f_l)$, the greatest common divisor of the elements $f_1,\dots, f_l$.

\subsection{Pseudo-isomorphism} Reference for this part is \cite{Bourbaki,Hillman}.

An $\cR$-module $N$ is {\em pseudo-zero} if for every prime ideal $\mathcal P$ of height 1, the localization $N_{\mathcal P}$ is 0. It is known that submodules and quotient modules of a pseudo-zero module are pseudo-zero.

In $\cR$, a prime ideal is of height 1 if and only if it is principal and generated by an irreducible polynomial.

An $\cR$-morphism $M_1\to M_2$ is a {\em pseudo-isomorphism} if  the kernel and co-kernel are pseudo-zero.

Two finitely generated torsion $\cR$-modules $M_1, M_2$ are pseudo-isomorphic if and only $\Delta_j(M_1) =\Delta_j(M_2)$ for every $j\ge 0$; in particular, a finitely generated torsion $\cR$-module is pseudo-zero if and only if $\Delta_0(M)=1$, see   \cite[Theorem 3.5]{Hillman}.

Let us formulate some well-known facts in the form that will be useful for us.

\begin{lemma} Suppose $I\subset \cR$ is a prime ideal, $I \neq \cR$.

a)
$\cR/I$ is pseudo-zero if and only if $I$ is not  principal.

b) If $\cR/I$ is pseudo-zero and $0\neq p\in I$,  then there is $q\in I$ such that $\gcd(p,q)=1$.

% (iii) There are non-zero $p,q\in I$ which are co-prime, i.e. $\gcd(p,q)=1$.
\label{lps}
\end{lemma}
\begin{proof} a) Since $I\neq \cR$ and $I$ is prime, $I=(p_1,\dots,p_l)$, where $p_i$'s are irreducible, non-unit, and distinct. One has
\begin{align*} \text{$\cR/I$ is pseudo-zero } & \Leftrightarrow \text {$\Delta_0(I)= \gcd(p_1,\dots,p_l)$ is 1}\\
& \Leftrightarrow l \ge 2\\
& \Leftrightarrow \text{$I$ is not principal}.
\end{align*}

b)  Suppose $q_1,\dots,q_l$ are all irreducible factors of $p$. Suppose the contrary that every $q\in I$ is not co-prime with $p$, i.e. every $q\in I$ is divisible by one of $q_i$'s.
Then $I \subset \cup_{i=1}^l (q_i)$. Since each ideal $(q_i)$ is prime,  there is an index $i$ such that $I \subset (q_i)$. Because $(q_i)$ has height 1 and $I$ is prime, this means
$I= (p_i)$, which is principal. This contradicts the fact that $\cR/I$ is pseudo-zero.
\end{proof}

% Let $M^* = \Hom_\cR(M,\cR)$ be the dual of $M$. There is a natural $\cR$-homomorphism from $M$ to its double dual $M^{**}$.

The following  is the main fact about pseudo-isomorphism which we will use.

\begin{thm}\cite[Theorem VII.4.5]{Bourbaki}  Any finitely generated module $\cR$-module $M$  is pseudo-isomorphic to $\Tor(M) \oplus M/\Tor(M)$.

% b) If $M$ is a torsion-free finitely generated module $\cR$-module, then the natural $\cR$-homomorphism $M \to M^{**}$ is a pseudo-isomorphism.
\label{ps}
\end{thm}

% Part (a) is a deep result, see \cite[Theorem VII.4.5]{Bourbaki}. Part (b) is easier, see \cite[Example VII.4.4.3]{Bourbaki}.

\begin{remark} It follows from \eqref{e001} that if $M_1,M_2$ are pseudo-isomorphic, then they have the same entropy, $h(M_1) = h(M_2)$. In particular, if
$M$ is pseudo-zero, then  $h(M)=0$.
\end{remark}
\subsection{Equivalence of Theorem \ref{T1} and Theorem \ref{T3}} Recall that $\BS$ is the unit circle in $\BC$. With the usual multiplication $\BS$ is an abelian Lie group. For an abelian group $G$, the Pontryagin dual $\hat G= \Hom(G,\BS)$ is a compact group. If $G\cong \BZ^k$, then $\hat G\cong \BS^k$. On the other hand, if $|G| <\infty$, then $G \cong \hat G$.
If  $G \cong \BZ^k \oplus \Tor_\BZ(G)$, then $\hat G \cong \BS^k \times \widehat{  \Tor_\BZ(G)}$. In particular,
the cardinality $|\Tor_\BZ G|$ of the $\BZ$-torsion  $G$ is the number of connected components of the compact group $\hat G$.

Suppose $M$ is a finitely generated $\cR$-module, and $\Gamma \subset \BZ^n$ a subgroup of finite index. By definition
\begin{align*} \Fix_\Gamma(\hat M) & = \{ x\in \hat M \mid \gamma\cdot x = x \quad \forall \gamma \in \Gamma\}\\
&= \{ x\in \hat M = \Hom(M,\BS) \mid x(y)  = x(\gamma(y))  \quad \forall \gamma \in \Gamma, y\in M  \}\\
&= \{ x\in \Hom(M,\BS) \mid x\big( (1-\gamma)(y)\big) =1 \quad \forall \gamma \in \Gamma, y\in M  \}\\
&= \{ x\in \Hom(M,\BS) \mid x\big( I(\Gamma)\,  M \big) =1 \}.
\end{align*}
It follows that
$$\Fix_\Gamma(\hat M) \cong \left( M/I(\Gamma)\,  M\right)^\wedge  \cong  \left(M \otimes \BZ[A_\Gamma]\right)^\wedge.$$
 We can conclude that $P_\Gamma(\hat M)$, the number of connected components of $\Fix_\Gamma(\hat M)$, is
\be  P_\Gamma(\hat M) = |\Tor_\BZ(M \otimes \BZ[A_\Gamma])|.
\label{e002}
\ee
From \eqref{e001} and \eqref{e002} we see that Theorem \ref{T1} and Theorem \ref{T3} are equivalent.

\subsection{Theorem \ref{T3}, the case when $M$ is a torsion module} As explained in Introduction, Theorem \ref{T3}, in the case when $M$ is a torsion module, has been proved \cite[Theorem 21.1]{Schmidt}. We will use this result for the case $N$ is pseudo-zero. Since $\Delta_0(N)=1$ if $N$ is pseudo-zero, we have the following.
\begin{proposition} Suppose $N$ is pseudo-zero. Then $N \sim 0$, i.e. $|\Tor_\BZ(N \otimes \BZ[A_\Gamma])| \sim 1$.
\label{pT3ps}
\end{proposition}

\subsection{Lattices in Hermitian spaces and $\BZ$-torsion}

Suppose $W$ is a finite-dimensional {\em based Hermitian space}, i.e. a $\BC$-vector space equipped with an Hermitian product $(.,.)$ and a  preferred  orthonormal basis. The $\BZ$-submodule $\Lambda
\subset W$ spanned by the basis is called the {\em fundamental lattice}.

For a $\BZ$-submodule (also called a lattice) $\Theta \subset \Lambda$ with $\BZ$-basis $v_1,\dots,v_l$ define
$$ \vol(\Theta) = |\det \left ( (v_i,v_j)_{i,j=1}^l \right)|^{1/2}.$$
It is clear that  $\vol(\Theta) \ge 1$.

For a lattice $\Theta \subset \Lambda$ define its orthogonal complement in $\Lambda$ by
$$ \Theta ^\perp = \{ x \in \Lambda \mid (x,y) = 0 \quad \forall y \in \Theta\}.$$
It is clear that $ \Theta \subset \Theta ^{\perp \perp}$. A lattice $\Theta$ is {\em primitive} is $\Theta = \Theta ^{\perp \perp}$. It is known that $\Theta$ is primitive if and only
it is cut out by a subspace, i.e.  $\Theta = (\Theta \otimes_\BZ \BQ) \cap \Lambda$; and if  $\Theta$ is primitive, then (see e.g. \cite{Bertrand})
\be
 \vol(\Theta)^2 = |\Lambda/ (\Theta \oplus \Theta^\perp)|.
 \label{e404}
 \ee

\begin{lemma}
For $i=1,2$ let $W_i$ be a finite-dimensional based Hermitian space with fundamental lattice $\Lambda_i$. Suppose $f: \Theta_1 \to \Theta_2$ is a $\BZ$-linear map, where $\Theta_i \subset \Lambda_i$ is
a lattice of maximal rank.
Then
$$ |\Tor_\BZ(\coker f) | \le || f||^{\rk \Theta_2} \, \vol(\Theta_1).$$
Here $||f||$ is the norm of the linear extension of $f$ to a $\BC$-linear operator from $W_1$ to $W_2$.
\label{l30}
\end{lemma}
\begin{proof} Let
$\overline{f(\Theta_1)}= (f(\Theta_1)\otimes _\BZ \BQ) \cap \Lambda_2$. Then
$$ \Tor_\BZ(\coker f) = \overline{f(\Theta_1)}/f(\Theta_1).$$
Hence
\begin{align} \left |\Tor_\BZ(\coker f)\right| & = |\overline{f(\Theta_1)}/f(\Theta_1)| \notag \\
&= \vol(f(\Theta_1) /\vol( \overline{f(\Theta_1)}).\label{e2001}
\end{align}

It is known that
\be  {\det}'(f) \, \vol(\Theta_1) = \vol(\ker f) \, \vol(f(\Theta_1)),
\label{e23}
\ee
where ${\det}'(f)$ is the product of all non-zero singular values of $f$, i.e. the square root of the product of all non-zero eigenvalues of $f^* f$.
From \eqref{e23} and \eqref{e2001} we have
\be
 \left |\Tor_\BZ(\coker f)\right| = \frac{{\det}'( f)\, \vol(\Theta_1)}{\vol(\ker f) \, \vol (\overline{f(\Theta_1)}) } \le  {\det}' (f)\, \vol(\Theta_1).
 \label{e302}
 \ee

The maximal singular value of $f$ is equal to $||f||$. The number of non-zero-singular values is less than or equal to the rank of $f$. Hence ${\det}'f$, being the product of the non-zero singular value, is  $\le ||f||^{\rk L_2}$.
Now from \eqref{e302} we get the lemma.
 \end{proof}

\subsection{Decomposition of the group ring of a finite abelian group}

\label{ss4}

\subsubsection{Decomposition over $\BC$}  Suppose $A$ is a finite  abelian group. The group ring $\BC[A]$ is a $\BC$-vector space of dimension $|A|$.
Equip $\BC[A]$ with a Hermitian product so that $A$ is an orthonormal basis. Then the integral group ring $\BZ[A]$  is the corresponding fundamental lattice.

The theory of representations of $A$ over $\BC$ is easy: $\BC[A]$ decomposes as a direct sum of mutually orthogonal one-dimensional $A$-modules:

\be \BC[A] = \bigoplus_{\chi \in \hat A} \BC e_\chi,
\label{e13}
\ee
where  $e_\chi$ is the idempotent
$$ e_\chi = \frac{1}{|A|} \sum_{a \in A} \chi(a^{-1} ) a.$$

The vector subspaces $\BC e_\chi$'s are not only orthogonal with respect to the Hermitian structure, but also orthogonal with respect to the ring structure in the sense that $e _\chi \, e_{\chi'} =0$ if $\chi \neq \chi'$. Each $\BC e_\chi$ is an ideal of the ring $\BZ[A]$.

For a $\BZ$-submodule   $X \subset \BZ[A]$ let $X_\BC$ the $\BC$-vector space spanned by $X$.

\subsubsection{Decomposition corresponding to a subgroup} The integral group ring $\BZ[A]$ does not have as nice a decomposition as \eqref{e13}. Given a subgroup $B \subset A$, we decompose
a subring of $\BZ[A]$ as follows.

 The natural projection $A \to A/B$ gives rise to the exact sequence
\be 0 \to \beta(B) \to \BZ[A] \to \BZ[A/B] \to 0,
\label{e1111}
\ee
where $\beta(B)$ is the ideal of $\BZ[A]$ generated by $1-b, b\in B$. As a lattice of $\BZ[A]$, $\beta(B)$ is primitive.

Let $\al(B)$ be the annihilator of $\beta(B)$:
$$\al(B) = \{ x\in \BZ[A] \mid x y =0 \quad \forall y \in \beta(B)\}.$$
Then  $\al(B)$ is also the orthogonal complement of $\beta(B)$ in $\BZ[A]$.
It is known that $\al(B)$ is the principal ideal generated by $u= u_B := \sum_{b\in B} b$, see eg. \cite{Bass}.

The complexification $\al_\BC(B)$ and $\beta_\BC(B)$ are easy to describe.
Tensoring \eqref{e1111} with $\BC$,
$$ 0 \to \beta_\BC(B) \to \BC[A] \to \BC[A/B] \to 0.
$$

 As a $\BC[A]$-module, $\BC[A/B]$ is isomorphic to $\al_\BC(B) = \beta_\BC(B)^\perp$, and

\be
\al_\BC(B)=\beta_\BC(B) ^\perp = \bigoplus_{\chi \in \hat A, \chi|_B=1} \BC e_\chi
\label{e34}
\ee
\be \rk_\BZ(\al(B))  = \dim_\BC(\al_\BC(B)) = |A|/|B|.
\label{e36}
\ee

\begin{proposition} % Suppose $B$ is a supgroup of a finite abelian group $A$.

 The finite group $\BZ[A]/\big(\al(B) \oplus \beta(B)  \big)$
 has order $|B|^{|A|/|B|}$.

%  c) One has $\vol(\al) = \vol(\beta) =
\label{p15}
\end{proposition}

\begin{proof} Let $y_1, \dots, y_\ell \in A$ be representatives of cosets of $B$ in $A$. Then $\ell = |A/B|$, and
the elements $y_1\,  u_B, \dots, y_\ell \, u_B$ form a $\BZ$-basis of $\al(B)$. It is easy to see that $(y_i \, u_B, y_j u_B)=0$ if $i\neq j$.
The length of each vector $y_j u_B$ is $(\sum_{b \in B} ||y_i  b||^2)^{1/2} = |B| ^{1/2}$. It follows that
$$\vol(\al(B)) = |B|^{\ell/2}.$$
From \eqref{e404} we have $|\BZ[A]/\big(\al(B) \oplus \beta(B)  \big)|= |B|^\ell = |B|^{|A/B|}$.
\end{proof}

\subsubsection{Decomposition corresponding to a collection of subgroups} \label{def}
Suppose $B_1,\dots, B_k$ are subgroups of a finite abelian group $A$. Let
$$ \al(B_1,\dots,B_k) = \sum_{j=1}^k \al(B_j), \quad \beta(B_1,\dots,B_k) = \bigcap_{j=1}^k \beta(B_j).$$

Then $\al(B_1,\dots,B_k)$ and $\beta(B_1,\dots,B_k)$ are primitive lattices of $\BZ[A]$, and they are orthogonal complement of each other in $\BZ[A]$.
In addition, both $\al(B_1,\dots,B_k)$ and $\beta(B_1,\dots,B_k)$ are ideals of $\BZ[A]$, and they are the annihilator of each other.

%% Then, over $\BC$,   we also have
%% $$ \al_\BC (B_1,\dots,B_k)= \sum_{j=1}^k \al_\BC(B_j), \quad \beta_\BC(B_1,\dots,B_k) = \bigcap_{j=1}^k \beta_\BC(B_j).$$

%% Since $\al_\BC(B_j) = \beta_\BC(B_j)^\perp$, we have

%% $$ \al_\BC(B_1,\dots, B_k)^\perp = \beta_\BC(B_1,\dots, B_k).$$

% One has $$ \beta_\BC(B_1,\dots, B_k) = \bigoplus_{\chi \in \cap_{j=1}^k S(B_j)^\perp} \BC e_\chi.$$

\begin{proposition} The finite group $\BZ[A]/\big( \beta(B_1,\dots, B_k) \oplus \al(B_1,\dots,B_k)\big)$ has order less than or equal to $\prod_{j=1}^k |B_j|^{|A|/|B_j|}$. Equivalently,
$$ \vol(\al(B_1,\dots,B_k)) \le \left ( \prod_{j=1}^k |B_j|^{|A|/|B_j|}\right)^{1/2}.$$
\label{p50}
 \end{proposition}

%\begin{lemma} The module $Q(p;\Gamma)= \BZ[A]/(\al \oplus \beta)$ is finite and has negligible growth, i.e. $| Q(p;\Gamma)| \sim 1$. It follows that $\vol(\al)$ has negligible growth.
% \label{l5}
%\end{lemma}

\begin{proof}  We write $\al= \al(B_1,\dots,B_k)$ and $\beta= \beta(B_1,\dots,B_k)$. Recall that $\beta = \cap_{j=1}^k \beta(B_j)$. We have

$$ \BZ[A]/(\al + \beta) \cong (\BZ[A]/\al)/\beta = (\BZ[A]/\al)/ \left(  \cap_{j=1}^k \beta(B_j) \right)%\hookrightarrow \prod_{j=1}^k (\RR/\al)/  \beta(B_j)
.$$

Since $(\BZ[A]/\al)/ \left(  \cap_{j=1}^k \beta(B_j)\right)$ injects in $ \prod_{j=1}^k (\BZ[A]/\al)/  \beta(B_j) = \prod_{j=1}^k (\BZ[A]/\beta(B_j) )/ \al  $, we have

\be
 \left |\BZ[A] /(\al + \beta) \right| \le \prod_{j=1}^k \left | (\BZ[A]/\beta(B_j) )/ \al \right|.
 \label{e11}
 \ee

Since $\al(B_j) \subset \al$,  $(\BZ[A]/\beta(B_j) )/ \al(B_j)  $ surjects onto $(\BZ[A]/\beta(B_j )/ \al $, hence
\be
  \left|  (\BZ[A]/\beta(B_j )/ \al \right| \le  \left| (\BZ[A]/\beta(B_j) )/ \al(B_j) \right |= \left | \BZ[A]/(\al(B_j) + \beta(B_j)) \right|.
\label{e16}
\ee
Inequalities \eqref{e11} and \eqref{e16}, together with Proposition \ref{p15}, show that

$$ \left | \BZ[A]/(\al \oplus \beta)\right|  \le \prod_{j=1}^k |B_j|^{ |A|/|B_j|}.$$
The equivalence between the two statements follows from \eqref{e404}.
\end{proof}

\subsection{Torsion points in algebraic varieties} We recall well-known facts about algebraic subgroups of $(\BC^*)^n$.

\subsubsection{Algebraic subgroups of $(\BC^*)^n$} With respect to the usual multiplication $\BC^* := \BC \setminus \{0\}$ is an abelian group, and so is $(\BC^*)^n$.
An {\em algebraic subgroup of $(\BC^*)^n$} is a subgroup which is closed in the Zariski topology.

For a lattice, i.e. a subgroup,  $\Lambda$ of $\BZ^n$, not necessarily of maximal rank,  recall that $I(\Lambda)$ is the ideal of $\cR$ generated by $ 1- t^\bk, \bk\in \Lambda$. Let
$G(\Lambda) = V_{I(\Lambda) }$,  the zero-set of $I(\Lambda) $, i.e. the set of all  $\bz \in \BC^n $ such that $\bz^\bk-1=0$ for every $\bk \in \Lambda$. Here for $\bk=(k_1,\dots,k_n) \in \BZ^n$ and $\bz=(z_1,\dots,z_n)\in (\BC^*)^n$ we set
$t^\bk = \prod_{i} t_i^{k_i}$ and $\bz^\bk = \prod_i z_i^{k_i}$.

It is easy to see that $G(\Lambda)$ is an algebraic subgroup.  The converse holds true: Every algebraic subgroup is equal to $G(\Lambda)$ for some lattice $\Lambda$, see  \cite{Schmidt_a}.

Every element $\bz \in G(\Lambda)$ defines a character $\chi_\bz$ of the quotient group $A_\Lambda:= \BZ/\Lambda$ via
$$ \chi_\bz(t^\bk) = \bz^\bk,$$
and conversely, every character of $A_\Lambda$ arises in this way.
Thus one can identify $G(\Lambda)$ with  $\Hom(A_\Lambda,\BC^*)$ via $\bz \to \chi_\bz$. We will write $e_\bz$ for the idempotent $\chi_{\chi_\bz}$, and the decomposition \eqref{e13}, with $\Lambda$ having maximal rank, now becomes

\be \BC[A_\Lambda] = \bigoplus_{\bz \in G(\Lambda)} \BC \,e_\bz.
\label{e133}
\ee

\def \cU {\mathbb U}
\subsubsection{Torsion points} A point $\bz \in (\BC^*)^n$ is a {\em torsion point} if it is a torsion element of the multiplicative group $(\BC^*)^n$.   Let $\cU$ denote the set of all roots of unity in  $\BC^*$. Then the set of torsion points of $(\BC^*)^n$ is $\cU^n$. For example, if $\Gamma \subset \BZ^n$ is a lattice of maximal rank, then $G(\Gamma)\subset \cU^n$.

The following fact is well-known in the theory of torsion points on algebraic varieties.
\begin{proposition} Suppose  $X \neq \BC^n$ is an algebraic subset of $\BC^n$ defined over $\BQ$. There exist a finite number of non-zero lattices $\Lambda_1,\dots, \Lambda_k$  in $\BZ^n$,  such that
$ \cU^n  \cap X \subset \cU^n \cap \big(\cup_{j=1}^kG(\Lambda_j)   \big)$, i.e.
any torsion point in $X$ belongs to $\cup_{j=1}^kG(\Lambda_j)$.
\label{p10}
\end{proposition}

\begin{proof} A {\em torsion coset} is a coset $uG$, where $u$ is torsion point and $G$ is an algebraic subgroup of $(\BC^*)^n$. It is well known that there is a finite number of torsion
cosets $u_j G_j \subset X$ such that every torsion point in $X$ belongs to $\cup_j u_j G_j$, see  \cite{Laurent,Schmidt_a}. Since $u_j G_j \subset X$, the dimension of $G_j$ is at most $n-1$. Let $U_j$ be the finite cyclic group generated by $u_j$. Then $U_jG_j$ is also an algebraic group of dimension $ \le n-1$. Hence $U_j G_j = G(\Lambda_j)$, with $\Lambda_j$ a non-zero lattice. Since $u_jG_j \subset U_jG_j$, it is clear that
every torsion point in $X$ belongs to $\cup_jG(\Lambda_j)$.
\end{proof}

\subsection{Elementary bounds from exact sequences}

\begin{lemma}Suppose $M$ is a finitely-generated $\cR$-module with a free resolution

$$ \dots \to \cR^{m_2} \overset{\partial_2}{\longrightarrow} \cR^{m_1} \overset{\partial_1}{\longrightarrow} \cR^{m_0} \to M \to 0,,$$
and $Q$ is a $\cR$-module with $|Q| < \infty$. Then $|{\operatorname{Tor}}_i^{\cR}(M, Q)| \le |Q|^{m_i}$ for every $i=0,1,\dots$.
\label{l7}
\end{lemma}
\begin{proof} By definition, ${\operatorname{Tor}}_i^{\cR}(M, Q)$ is the homology groups of the complex
$$ \dots \to \cR^{m_2}\otimes Q  \to\cR^{m_1}\otimes Q \to \cR^{m_0}\otimes Q \to 0 .$$
 Since the $i$-th term of the this complex is $\cR^{m_i} \otimes Q \cong Q^{m_i}$, a finite group of order $|Q|^{m_i}$, its $i$-th homology group has $\le |Q|^{m_i}$ elements.
\end{proof}

\begin{lemma}
 Suppose in  an exact sequence of abelian groups
$$ \dots  \to N_1(\Gamma) \to M_1(\Gamma) \to M_2(\Gamma) \to N_2(\Gamma) \to \dots $$
 each $M_i(\Gamma),N_i(\Gamma)$ is an abelian group depending on subgroups $\Gamma \subset \BZ^n$ of finite index. Assume further that
$N_1(\Gamma)$ and $N_2(\Gamma)$ are finite, and
$$ |N_1(\Gamma)| \sim 1 \sim |N_2(\Gamma)|.$$
Then
$$|\Tor_\BZ(M_1(\Gamma))| \sim |\Tor_\BZ(M_2(\Gamma))|.$$
\label{l8}
\end{lemma}

\begin{proof} Replacing $N_1$ by an appropriate quotient and $N_2$ by an appropriate subgroup, we may assume that
$$ 0 \to N_1(\Gamma) \to M_1(\Gamma) \to M_2(\Gamma) \to N_2(\Gamma) \to 0 $$
is exact. We then have

\be M_1/N_1 \hookrightarrow M_2 \twoheadrightarrow N_2. \label{e20}
\ee
 The inclusion in \eqref{e20} shows that
$$ |\Tor_\BZ (M_1)| /|N_1| \le |\Tor_\BZ(M_2)|,$$
and the surjecion in \eqref{e20} shows that
$$ \Tor_\BZ (M_2)  \le |N_2| \,  |\Tor_\BZ(M_1/N_1)| = |N_2| \,  |\Tor_\BZ(M_1)| /|N_1|.$$
From there we get the conclusion of the lemma.
\end{proof}

\begin{lemma}
 Suppose $\fB$ is a free abelian group of finite rank, and $p,q: \fB \to \fB$ are $\BZ$-linear operators with $p$ injective. Let $\cC$ be the complex
$$ 0 \to \fB \overset {\partial_2} {\longrightarrow} \fB^2 \overset {\partial_1} {\longrightarrow} \fB \to 0,$$
where $\partial_2(a) = (-q (a ), p(a))$, $\partial_1(a,b) = p(a) + q(b)$.
Then the homology groups of $\cC$ are finite, and
$$ |H_1(\cC) | = |H_0 (\cC)|.$$
\label{l90}
\end{lemma}
\begin{proof} The complex $\cC$ is the middle row of  the commutative diagram with exact columns

$$\begin{CD} 0  @ >>>    0     @ >>>          \fB @ > p >>    \fB @ >>> 0 \\
@VVV @ VVV @VV i_1V  @VVV @VVV \\
           0   @>>>    \fB   @ > \partial_2>>  \fB \oplus  \fB @  > \partial_1 >> \fB @ >>> 0 \\
           @VVV @ VVV @VV i_2V  @VVV @VVV \\
           0 @ >>>          \fB @ > p >>    \fB @ >>> 0 @ >>> 0
\end{CD}.
$$
Here $i_1(a)= (a,0), i_2(a,b)=b$.
Let $\cC_1$ be the first row and $\cC_2$ the last row. Then $0 \to \cC_1 \to \cC\to \cC_2\to 0$ is exact. The long exact sequence, together with $H_1(\cC_1)= H_0(\cC_2)$, gives us
the following exact sequence
$$ 0 \to H_1(\cC) \to H_1(\cC_2) \to H_0(\cC_1) \to H_0(\cC) \to 0.$$
Note that $|H_1(\cC_2)|= |H_0(\cC_1)| =|\coker p|$, which is finite since $p$ is injective. It follows that  $H_1(\cC)$ and $H_0(\cC)$ are finite. In an exact sequence of finite abelian groups,
the alternating product of the cardinalities is 1. Hence, with the two middle terms having $|H_1(\cC_2)|= |H_0(\cC_1)|$, we must have $|H_1(\cC)| =| H_0(\cC)|$.
% \begin{pmatrix} F\\ G \end{pmatrix}
\end{proof}

\section{Approximation of $\BZ[\BZ^n/\Gamma]$} \label{approx}

\subsection{Approximation of $\BZ[A_\Gamma]$: Formulation of results} As mentioned in the introduction, we search for a good approximation of $\BZ[\BZ^n /\Gamma]$ as $\la \Gamma \ra \to \infty$. The approximation  depends on some extra choice, namely, a non-zero element $p \in \cR$.

Fix a {\em  non-zero} Laurent polynomial $p\in \cR= \BZ[t^{\pm 1},\dots,t^{\pm n}]$. For each subgroup $\Gamma \subset \BZ^n$ of rank $n$ we will construct an  $\cR$-module $\beta(p;\Gamma)$ with the following properties.

\begin{proposition} \label{mainx}

(i) For every finitely generated $\cR$-module $M$ one has
$$ |\Tor_\BZ (M \otimes \BZ[A_\Gamma] )|  \sim | \Tor_\BZ (M \otimes \beta(p;\Gamma) )  |.$$

(ii) Suppose $p\in I$, where $ I \neq \cR$ is ideal of $\cR$ such that $\cR/I$ is pseudo-zero.
Then for each $i=0, 1$ the module ${\operatorname{Tor}}^{\cR}_i(\beta(p;\Gamma) , \cR/I)$ is finite, and
$$ | {\operatorname{Tor}}^{\cR}_i(\beta(p;\Gamma) , \cR/I)| \sim 1.
$$
\end{proposition}

The remaining part of this section is devoted to the construction of $\beta(p;\Gamma)$ and the proof of Proposition \ref{mainx}.

\subsection{Heuristics} There is no rigorous mathematics in this subsection. Logically the reader can skip this subsection.

In the estimate of $\Tor_\BZ(M\otimes \BZ[A_\Gamma])$ using exact sequences, finiteness is very helpful. We will try to decompose $\BZ[A]$ as a sum of two
submodules, one is negligible, and the other if finite if tensoring with pre-given modules.

We have the decomposition
\eqref{e133} of $\BC[A_\Gamma]$ into irreducible components

$$ \BC[A_\Gamma] = \bigoplus_{\bz \in G(\Gamma)} \BC \, e_\bz. $$

The module $M \otimes \BC[A_\Gamma]$ will decompose accordingly.
Albeit over $\BC$, this gives us hint that some   $\bz \in G(\Gamma)$ are ``good" and some are ``bad".
Here a good $\bz$  must satisfy some non-degeneracy property, and if a point is good, all its Galois conjugates are good. Combining  all good points together one should
get some ``integral" sub-module of $\BZ[A_\Gamma]$ for which
non-degeneracy conditions imply some kind of finiteness.
 If $S$ is the set of all bad points, and $S^\perp$ be its complement in $G(\Gamma)$, then one has
$$ \BC[A_\Gamma] = \left (\bigoplus_{\bz \in S} \BC \, e_\bz \right) \oplus \left (\bigoplus_{\bz \in S^\perp} \BC \, e_\bz \right).$$
The module $\beta$ would be the ``integral spine" of the second part.

The set of bad points will consists of those in $G(\Gamma)$ which are zeros  $p$. For good points $\bz$, $p(\bz)\neq 0$, and this will give us the non-degeneracy condition. We control the set of bad points, which is the intersection $G(\Gamma) \cap V_p$,
by using theory of torsion points on $V_p$, see Proposition \ref{p10}.

\subsection{Definition of $\beta(p;\Gamma)$}  The zero set
$$ V_p: = \{ (z_1,\dots,z_n) \in (\BC^*)^n \mid p(z_1,\dots,z_n)= 0 \}$$
is an algebraic subset of $(\BC^*)^n$ of dimension $\le n-1$.
Let $\Lambda_1,\dots, \Lambda_k$ be the non-zero subgroups of $\BZ^n$  described in Proposition \ref{p10} with $X= V_p $.
By construction,

\be
 \text{ if a torsion point $\bz$ does not belong to $\bigcup_{j=1}^k G(\Lambda_j)$, then $p(\bz) \neq 0$.}
 \label{e5001}
 \ee

Suppose $A=A_\Gamma:= \BZ^n /\Gamma$, where $\Gamma \subset \BZ^n$ is a subgroup of maximal rank $n$. The abelian groups
$B_j = (\Lambda_j + \Gamma)/\Gamma$ are subgroups of $A= \BZ^n /\Gamma$. Let
$$\al(p;\Gamma) = \al(B_1,\dots,B_k), \quad \beta(p;\Gamma) = \beta(B_1,\dots, B_k),$$
where $\al(B_1,\dots,B_k)$ and $\beta(B_1,\dots, B_k)$ are ideals of $\BZ[A]$ defined as in Section \ref{def}.

We partition $G(\Gamma) = \hat A$ into two disjoint subsets $S, S^\perp$ by
\be
S = G(\Gamma) \bigcap \left( \bigcup_j G(\Lambda_j) \right), \quad S^\perp = G(\Gamma) \setminus S.
\label{e6001}
\ee
We will see that as $\la \Gamma \ra \to \infty$, $S$ is small compared to its complement $S^\perp$.
Note that $\chi_\bz$, with $\bz \in G(\Gamma)$, takes value 1 on $B_j$ exactly when $\bz \in G(\Lambda_j)$. Hence from \eqref{e34} we have
$$\al_\BC (B_j) = \bigoplus_{\bz \in G(\Gamma) \cap G(\Lambda_j)} \BC _\bz,$$
and hence

\be
\beta_\BC(p;\Gamma) = \bigoplus_{\bz \in  S^\perp} \BC _\bz,\quad \al_\BC(p;\Gamma) = \bigoplus_{\bz \in  S} \BC _\bz.
\ee

We will write $\al= \al(p;\Gamma), \beta = \beta(p;\Gamma)$. Let $\pr: \cR \to \BZ[A_\Gamma]$ be the canonical projection.
Note that $\pr^{-1}(0)$ is the
ideal of all polynomials taking values 0 at every point of $G(\Gamma)$. Similarly,
$\tal= \pr^{-1}(\al)$ is the
ideal of all polynomials taking values 0 at every point of $S^\perp$. Over $\BC$,
  $\tal_\BC$ is   the
reduced ideal of $\cR_\BC=\BC[t_1^{\pm1}, \dots, t_n^{\pm1}]$ whose zero set is $S^\perp$, $V_{\tal_\BC}= S^\perp$. In addition,

\be \cR_\BC /\tal_\BC \cong \BC[A]/\al_\BC \cong \beta_\BC.
\ee

The important facts concerning $\al(p;\Gamma)$ and $\beta(p;\Gamma)$  are the following.
\begin{lemma} \label{l5}

a)  $S^\perp= V_{\tal_\BC}$ does not intersect $V_p$. It follows that the ideal of $\cR_\BC$ generated by $p$ and $\tal_\BC$ is the whole $\cR_\BC$.

b) The multiplication map  $p:\beta \to \beta, x \to p\cdot x$, is injective. It follows that ${\operatorname{Tor}}^\cR_1(\cR/(p), \beta)=0$.

c) The quotient group $Q(p;\Gamma) := \BZ[A_\Gamma]/\big(\al(p;\Gamma) \oplus \beta(p;\Gamma)\big)$ is finite and its order is negligible,
$|Q(p;\Gamma)| \sim 1$.

d) $|S|= \rk_\BZ \al(p;\Gamma)$ is small compared to the $\rk_\BZ \BZ[A_\Gamma]= |A_\Gamma|$ in the sense that
$$\lim_{\la \Gamma \ra \to \infty }\frac{\rk_\BZ \al(p;\Gamma)}{|A_\Gamma|} = 0  .$$

e) One has $\vol(\al) \sim 1$.
\end{lemma}

\begin{proof} a) Suppose $\bz \in S^\perp$. By definition \eqref{e6001}, $\bz$ is a torsion point not belonging to $\bigcup_{j=1}^k G(\Lambda_j)$. By \eqref{e5001}, $p(\bz)\neq 0$.
In other words,  $V_p \cap V_{\tal_\BC}=\emptyset$. By Nullstellensatz, the ideal  generated by $p$ and $\tal_\BC$ is the whole $\cR_\BC$.

b) Note that $\BC \, e_\bz$ is a $\cR_\BC$-module by the action $f\cdot e_\bz= f(\bz) \, e_\bz$. If $\bz \in S^\perp$, then $p(\bz) \neq 0 $, hence $p: \BC \, e_\bz \to \BC \, e_\bz$ is an isomorphism. Since $\beta_\BC = \bigoplus_{\bz \in S^\perp} \BC \, e_\bz$, the map $p : \beta_\BC \to \beta_\BC$ is also an isomorphism. It follows that $p :\beta \to \beta$ is injective.

One has ${\operatorname{Tor}}^\cR_1(\cR/(p), \beta)= \ker \left ( p:\beta \to \beta, x \to p\cdot x  \right)=0$.

c)
We will first show that
 for each $j=1,\dots, k$,
\be \lim_{\la \Gamma \ra \to \infty } |B_j| = \infty.
\label{e37}
\ee

 % We will prove the strong result $|B_j| > C\la \Gamma \ra$ for some constant $C >0$.

By definition, $B_j= (\Lambda_j +\Gamma)/\Gamma$.
Fix an element $x \in \Lambda_j$, $x \neq 0$, and look at the degree of $x$ in $B_j= (\Lambda_j +\Gamma)/\Gamma$.
If $m |x| < \la \Gamma \ra$, then $m|x|$ does not belong to $\Gamma$ by the definition of $\la \Gamma \ra$, and hence $mx$ is not 0 in $B_j= (\Lambda_j +\Gamma)/\Gamma$.
This means the cyclic subgroup of $B_j$ generated by $x$ has order at least $\la \Gamma \ra /|x|$.
It follows that $|B_j| \ge \la \Gamma \ra /|x|$.  Hence
$ \lim_{\la \Gamma \ra \to \infty } |B_j| = \infty$.

From Proposition \ref{p50},
$$ |Q(p;\Gamma)|^{1/|A|} \le \prod_{j=1}^k |B_j|^{1/|B_j|},$$
from which and \eqref{e37} we get $|Q(p;\Gamma)| \sim 1$.

d) By \eqref{e36} one has $\rk(\al(B_j))= |A|/|B_j|$. Since $\al = \sum \al(B_j)$, one gets
\be
 \rk (\al)/|A|  \le \sum_{j=1}^k \rk \al(B_j)/|A| = \sum _{j=1}^k (1/|B_j|),
 \ee
which, with \eqref{e37}, shows that $\lim_{\la \Gamma \ra \to \infty }\frac{\rk_\BZ \al(p;\Gamma)}{|A_\Gamma|} = 0$.

e) This follows immediately  from \eqref{e404} and part (c).
\end{proof}

%%  \begin{remark} Part c) has the following interpretation.  \end{remark}

\subsection{Contribution from $\al(p;\Gamma)$ is negligible}

The ideals $\al$ and $\beta$, being $\BZ[A_\Gamma]$-module, can be naturally considered as $\cR$-modules.

\begin{lemma} Suppose $M$ is a finitely generated $\cR$-module. Then
$$ |\Tor_\BZ(M \otimes \al(p;\Gamma))| \sim 1.$$

\label{l9}
\end{lemma}
\begin{proof} Tensoring the presentation

$$
 \cR^{m_1} \overset{\mathcal \partial_1}{\longrightarrow} \cR^{m_0} \to M \to 0.
 $$
 with $\BZ[A]$ and $\al$ respectively, one gets

$$
 (\BZ[A_\Gamma])^{m_1} \overset{\partial_{1,\Gamma}}{\longrightarrow} (\BZ[A_\Gamma])^{m_0} \to M\otimes \BZ[A_\Gamma]\to 0,
 %\label{e31}
 $$
\be  \al ^{m_1} \overset{\partial_{1,\al}}{\longrightarrow} \al ^{m_0} \to M\otimes \al \to 0,
\label{e32}
\ee
with $\partial_{1,\al}$ the restriction of  $\partial_{1,\Gamma}$.

Recall that we have a Hermitian structure on  $\BC[A]$.
It is not difficult to find an upper bound, not depending on $\Gamma$, for all the operator $\partial_{1,\Gamma}$. In fact, by \cite[Lemma 2.5]{Luck},

$$ ||\cO_\Gamma || \le D := m_1 m_0 \max_{i,j}\{ |\cO_{ij}|_1 \},$$
where for a Laurent polynomial $a \in \BZ[t_1^{\pm 1}, \dots, t_n^{\pm 1}]$ the norm $|a|_1$ is the sum of the absolute values of its coefficients.

 Because $\al_\BC$ is an invariant subspace $\BC[A]$,  we also have
$$ ||\cO_\al|| \le D.$$

Applying Lemma \ref{l30} to the sequence \eqref{e32} we get
$$ |\Tor_\BZ( M \otimes \al)|\le    D ^{\rk_\BZ \al}\, \vol(\al).$$
The right hand side has negligible growth, by Lemma \ref{l5}(d) and (e).
\end{proof}

\subsection{Proof of Proposition \ref{mainx} part (i)}

\begin{proof} Recall that $A= A_\Gamma:= \BZ^n/\Gamma $. We have an exact sequence
\be
 0 \to (\al \oplus \beta) \to \BZ[A] \to Q \to 0,
 \label{e25}
 \ee
with $Q=|Q(p;\Gamma)|\sim 1$ by Lemma \ref{l5}a.  Tensoring \eqref{e25} with $M$,
\be
 \dots\to  {\operatorname{Tor}}_1^{\cR} (M,Q) \to \big ((M\otimes \al) \oplus (M\otimes  \beta)\big) \to M\otimes \BZ[A] \to {\operatorname{Tor}}_0^{\cR} (M,Q) \to 0.
 \label{e251}
 \ee
 Lemma \ref{l7} shows that $|{\operatorname{Tor}}_i^{\cR} (M,Q)|< |Q|^{m_i}$ for some constant $m_i$ depending on $M$ only. Since $Q \sim 1$, we also have
 \be
 |{\operatorname{Tor}}_i^{\cR} (M,Q)|\sim 1. \notag
 \ee
 Applying Lemma \ref{l8} to the sequence \eqref{e251}, we get

$$ |\Tor_\BZ(M\otimes \al)  \oplus \Tor_\BZ(M\otimes  \beta)|  \sim  |\Tor_\BZ (M\otimes \BZ[A])|. $$

Since $|\Tor_\BZ(M\otimes \al)| \sim 1$ by Lemma \ref{l9}, we have  $|\Tor_\BZ(M\otimes  \beta)|  \sim  |\Tor_\BZ (M\otimes \BZ[A])|$.
\end{proof}

\subsection{The intermediate ideal $J=(p,q)$}
To prepare for the proof of Proposition part (ii), we first study the ideal $J=(p,q)$, where $q\in \cR$ is co-prime with $p$. The reason is $\cR/J$ has a simple free resolution, and hence ${\operatorname{Tor}}_i^{\cR}\big (\cR/J, \beta(p;\Gamma)\big )$ is easy to study.

\begin{lemma} Both modules ${\operatorname{Tor}}_1^{\cR}\big (\cR/J, \beta(p;\Gamma)\big )$ and ${\operatorname{Tor}}_0^{\cR}\big (\cR/J, \beta(p;\Gamma)\big )$ are finite, have the same cardinality, and
have negligible growth, i.e.
\be
\left |{\operatorname{Tor}}_1^{\cR}\big (\cR/J, \beta(p;\Gamma)\big )\right| = \left|{\operatorname{Tor}}_0^{\cR}\big (\cR/J, \beta(p;\Gamma)\big )\right| \sim 1.
\label{e95}
\ee
\label{l93}
\end{lemma}
\begin{proof} We have the following free resolution of $\cR/J$
\be
 0 \to \cR \overset {d_2} {\longrightarrow} \cR \oplus \cR \overset {d_1} {\longrightarrow} \cR \to \cR/J \to 0,\label{e90}
 \ee
where $d_2(a) = (-q a, pa)$ and $d_1(a,b) = pa + qb$. This can be directly checked easily, or can be deduced from the theory of Koszul complex as follows. Since $p,q$ are co-prime, the sequence $(p,q)$ is a regular
sequence of $\cR$ (see Exercise 5 of page 102 of \cite{Kaplansky}). Hence the Koszul complex of $(p,q)$, which is \eqref{e90}, is a free resolution of $\cR/J$.

From the free resolution \eqref{e90}, ${\operatorname{Tor}}_i^{\cR}(\cR/J, \beta)$ is the $i$-th homology of the complex
\be  0  \to \beta \overset {(-q,p)} {\longrightarrow} \beta \oplus \beta \overset {\begin{pmatrix}p \\ q  \end{pmatrix}} {\longrightarrow} \beta \to 0.
\ee
The module $\beta$ is a free $\BZ$-module of finite rank, and  the map $p: \beta \to \beta$ is injective,  by Lemma \ref{l5}(b). From Lemma \ref{l90} we see that both ${\operatorname{Tor}}_1^{\cR}(\cR/J, \beta)$ and ${\operatorname{Tor}}_0^{\cR}(\cR/J, \beta)=  (\cR/J) \otimes \beta$ are finite, and
$$ |{\operatorname{Tor}}_1^{\cR}(\cR/J, \beta)| = | \cR/J \otimes \beta|.$$
By Lemma \ref{lps}, $\cR/J$ is pseudo-zero since $p$ and $q$ are co-prime. We have
\begin{align*}
| (\cR/J) \otimes \beta| & \sim | (\cR/J) \otimes \BZ[A_\Gamma]| %\quad
&\text{by Proposition \ref{mainx}(i)}\\
&\sim 1  %\quad
& \text{by Proposition \ref{pT3ps} }.
\end{align*}

This completes the proof of the lemma.
\end{proof}

\subsection{Complexification of $\Tor$ modules} Recall that $\cR_\BC = \BC[t_1^{\pm 1}, \dots, t_n^{\pm 1}]$. Observe that
 $$ \cR_\BC  \cong \cR \otimes_\BZ \BC, \quad \text{and} \quad \cR_\BC  \cong \cR \otimes_\cR \cR_\BC.$$

 Let $I_\BC$ be the $\BC$-span of $I$ in $\cR_\BC$. Then $I_\BC$ is also the extension of $I$ from $\cR$ to $\cR_\BC$.
\begin{lemma}As $\BZ$-modules, for every $i$,
$$ \big( {\operatorname{Tor}}_i^\cR(\cR/I,\beta) \big) \otimes_\BZ \BC \cong {\operatorname{Tor}}_i^{\cR_\BC}(\cR_\BC/I_\BC,\beta_\BC).$$

\label{l200}
\end{lemma}

\begin{proof}

Since  $\BC$ is flat over $\BZ$, we have
$(\cR/I) \otimes_\BZ \BC \cong \cR_\BC / I_\BC$.

Since $\cR_\BC$ is flat over $\cR$, we have $   (\cR/I)\otimes_\cR \cR_\BC  \cong \cR_\BC / I_\BC$. It follows that

\be
(\cR/I) \otimes_\BZ \BC \cong (\cR/I)\otimes_\cR \cR_\BC.
\label{e102}
\ee

Suppose $\cC \to \beta$ is a free resolution of $\beta$. By definition,
\be  {\operatorname{Tor}}_i^\cR(\cR/I, \beta) = H_i(\cC\otimes_\cR (\cR/I)).
\label{e103}
\ee
Tensoring \eqref{e103} with $\cR_\BC$, a flat $\cR$-module, we get

\be  {\operatorname{Tor}}_i^{\cR_\BC}(\cR_\BC/I_\BC,\beta_\BC) = H_i(\cC\otimes_\cR (\cR/I) \otimes_\cR \cR_\BC).
\label{e104}
\ee

Tensoring \eqref{e103} over $\BZ$ with  $\BC$, a flat $\BZ$-module, we get

\be \big( {\operatorname{Tor}}_i^\cR(\cR/I,\beta) \big) \otimes_\BZ \BC \cong H_i\Big( \big (\cC\otimes_\cR (\cR/I)\big) \otimes_\BZ \BC\Big ).
\label{e105}
\ee
Since $\cC$ is free, each term of $\cC$ is a direct of several $\cR$. It follows from \eqref{e102} that the right hand sides of \eqref{e104} and \eqref{e105} are
isomorphic as $\BZ$-modules, whence the lemma.
\end{proof}

\subsection{Proof of Proposition \ref{mainx} (ii)}

\begin{proof}

 a) The case $i=0$. Recall that ${\operatorname{Tor}}_0^{\cR}(\cR/I, \beta(p;\Gamma)) = (\cR/I) \otimes \beta$.

 Since $(p) \subset I$, we have a natural surjection
 $  \cR/(p)  \twoheadrightarrow \cR/I$.
 Tensoring with $\beta= \beta(p;\Gamma)$ we get a surjective map
  $$ \beta \otimes \big( \cR/(p)\big)   \twoheadrightarrow \beta \otimes (\cR/I).$$
  Now  $\beta \otimes \big( \cR/(p)\big) \cong \beta/p$, which is finite since $p$ acts on the finite-rank  free abelian group $\beta$ by an injection,
  see Lemma \ref{l5}(b). It follows that $\beta \otimes \cR/I$ is finite.

Since $\cR/I$ is pseudo-zero,  by Proposition \ref{mainx}(i) and  Proposition \ref{pT3ps},
$$|\beta \otimes_\cR (\cR/I)| \sim  |\BZ[A_\Gamma] \otimes_\cR (\cR/I)| \sim 1.$$

b) The case $i=1$.
First we show that ${\operatorname{Tor}}_1^{\cR}(\beta, \cR/I)$ is finite.

 By Lemma \ref{l5}(a), the $\cR_\BC$-ideal generated by $\tilde\al_\BC$ and $p$ is $\cR_\BC$,
hence $\tal_\BC + I_\BC = \cR_\BC$ because $ p \in I$. It is well-known  then (see eg.  \cite[Chapter 1]{AM})
 \be
 \tal_\BC \cap I_\BC =\tal_\BC \, I_\BC.
 \label{e290}
 \ee

 For two ideals $I_1, I_2$ in a commutative ring $R$, it is known that ${\operatorname{Tor}}_1^R(R/I_1,R/I_2) \cong I_1\cap I_2/I_1I_2$. Hence from \eqref{e290} we have
$${\operatorname{Tor}}_1^{\cR_\BC}(\cR_\BC/\tal_\BC, \cR_\BC/I_\BC)=0.$$

Since $\beta_\BC = \cR_\BC/\tal_\BC$, this can be rewritten as $$  {\operatorname{Tor}}_1^{\cR_\BC}(\beta_\BC, \cR_\BC/I_\BC) = 0,$$
which,
by Lemma \ref{l200}, implies that
\be
\big(  {\operatorname{Tor}}_1^{\cR}(\beta, \cR/I)\big)  \otimes _\BZ \BC =0. \label{e300}
\ee

Since ${\operatorname{Tor}}_1^{\cR}(\beta, \cR/I)$ is a finitely generated abelian group, \eqref{e300} is equivalent to the fact
${\operatorname{Tor}}_1^{\cR}(\beta, \cR/I)$ is finite.

Now we show that $|{\operatorname{Tor}}_1^{\cR}(\beta, \cR/I) |\sim 1$. Since $\cR/I$ is pseudo-zero,  $I \neq (p)$. This means there is $q \in I$ such that $q$ is not divisible by $p$. Since $p$ is irreducible,
$p$ and $q$ are co-prime. Let $J =(p,q)$. Then $(p) \subset J \subset I$.

Tensoring $\beta$ with the exact sequence
$$ 0 \to I/J \to \cR/J \to \cR/I \to 0$$ we
get the exact sequence
\be
 \dots \to {\operatorname{Tor}}_1^{\cR}(\cR/J, \beta) \to {\operatorname{Tor}}_1^{\cR}(\cR/I, \beta) \to \big( (I/J) \otimes \beta\big)\to \dots
 \label{e110}
 \ee
 The module $I/J$, being a submodule of the pseudo-zero module $\cR/J$, is also pseudo-zero. Hence by Proposition \ref{mainx}(i) and Proposition \ref{pT3ps},
 \be |\Tor_\BZ \big( (I/J) \otimes \beta\big)| \sim 1.
 \ee
 By Lemma \ref{l93}, ${\operatorname{Tor}}_1^{\cR}(\cR/J, \beta)$ is finite and has negligible growth,

 \be
  |{\operatorname{Tor}}_1^{\cR}(\cR/J, \beta)| \sim 1. \label{e1101}
 \ee

  The middle term of \eqref{e110}, being finite, must satisfy

 $$ |{\operatorname{Tor}}_1^{\cR}(\cR/I, \beta) | \le | {\operatorname{Tor}}_1^{\cR}(\cR/J, \beta) | \, |\Tor_\BZ \big( (I/J) \otimes \beta\big)|$$
 and hence by \eqref{e110} and \eqref{e1101} is negligible,  $ |{\operatorname{Tor}}_1^{\cR}(\cR/I, \beta)| \sim 1$.
\end{proof}

\section{Proof of Theorems \ref{T4}, \ref{T3}, and \ref{T33}}\label{S3}
\subsection{Pseudo-zero kernel}

\begin{lemma} Suppose $M_1, M_2$ are finitely-generated $\cR$-module, $I \subset \cR$ is a prime ideal such that $\cR/I$ is pseudo-zero, and
\be  0 \to \cR/I \to M_1 \to M_2\to 0 \label{e4}
\ee
is exact. Then $M_1 \sim M_2$.
\label{l906}
\end{lemma}

\begin{proof} Choose a non-zero irreducible $p \in I$ and let $\beta= \beta(p;\Gamma)$. Tensoring \eqref{e4} with $\beta$, we get the exact sequence

\be
 \dots \to  (\cR/I)\otimes \beta \to M_1 \otimes \beta  \to M_2 \otimes \beta \to 0.
 \label{e70}
 \ee

 By Proposition \ref{mainx}(ii), $(\cR/I)\otimes \beta $ is finite and
 $|(\cR/I)\otimes \beta | \sim 1$.

 Applying Lemma \ref{l8} to the sequence \eqref{e70}, we get
 $$ |\Tor_\BZ(M_1 \otimes \beta ) |  \sim |\Tor_\BZ( M_2 \otimes \beta ) |.$$
 By Proposition \ref{mainx}(i), $|\Tor_\BZ(M_i \otimes \beta ) | \sim |\Tor_\BZ(M_i \otimes \BZ[A_\Gamma] ) |$. Hence we can conclude that
 $$ |\Tor_\BZ(M_1\otimes \BZ[A_\Gamma]| \sim |\Tor_\BZ(M_2\otimes \BZ[A_\Gamma]|.$$
This means $M_1 \sim M_2$.
\end{proof}

\begin{lemma} Suppose $N, M_1$ and $M_2$ are finitely generated $\cR$-modules, and $N$ is pseudo-zero.
If
\be  0 \to N \to M_1 \to M_2\to 0 \label{e41}
\ee
 is exact, then $M_1 \sim M_2$.
 \label{l908}
\end{lemma}

\begin{proof}
It is well-known that
there is a composition series
\be N = N_s \supset N_{s-1} \supset \dots \supset N_1 \supset N_0=0
\label{e71}
\ee
 such that for each $i$,  $N_{i+1}/N_i \cong
\cR/I_i$ for some prime ideal $I_i$, see eg. \cite[Theorem IV.4.1]{Bourbaki}.  We use induction on $s$. The case $s=1$ has been proved, see Lemma \ref{l906}.

Let $M'_1= M_1/N_{s-1}$ and $N':= N/N_{s-1}\cong \cR/I$, with $I=I_{s-1}$. From \eqref{e41} we have
\begin{align}
 0  & \to N'  \to M_1' \to  M_2  \to  0.
 \label{eq.41a}
 \end{align}
%$\cong \cR/I$, with $I=I_{s-1}$. 

From $M'_1= M_1/N_{s-1}$, we have
\begin{align}
0  & \to N_{s-1} \to M_1  \to  M_1'  \to  0. \label{eq.41b}
\end{align}
 
 Note that $N'$ and $N_{s-1}$, being either a quotient or a submodule  of the pseudo-zero module $N$, are pseudo-zero.
By induction and the case $s=1$,  from the exact  sequences \eqref{eq.41a} and \eqref{eq.41b}, we have

$$ M_1' \sim M_2, \quad
M_1 \sim M_1'.$$

Hence $ M_1 \sim M_2$.
\end{proof}

\subsection{Pseudo-zero quotient}
\begin{lemma} Suppose $N, M_1$ and $M_2$ are finitely generated $\cR$-modules, and $N$ is pseudo-zero.
 If
\be  0 \to M_1 \to M_2 \to N\to 0 \label{e5}
\ee is exact, then  $M_1 \sim M_2$.
\label{mainn}

\end{lemma}

\begin{proof}
  Again using induction on the length  of the composition series \eqref{e71} like in the proof of Lemma \ref{l908} we can assume that $N =\cR/I$, where $I\neq \cR$ is a prime module.
Choose a non-zero irreducible $p \in I$ and let $\beta= \beta(p;\Gamma)$.

 Tensoring \eqref{e5} with $\beta$, we have

\be
 \dots \to {\operatorname{Tor}}_1^{\cR}(\cR/I, \beta )\to    M_1 \otimes \beta  \to M_2 \otimes \beta \to {\operatorname{Tor}}_0^{\cR}(\cR/I, \beta )  \to 0.
 \label{e711}
 \ee
By Proposition \ref{mainx}(ii), ${\operatorname{Tor}}_i^{\cR}(\cR/I, \beta )$ is finite and
 $|{\operatorname{Tor}}_i^{\cR}(\cR/J, \beta )| \sim 1$ for $i=0,1$.

 Applying Lemma \ref{l8} to the sequence \eqref{e71}, we get
 $$ |\Tor_\BZ(M_1 \otimes \beta ) |  \sim |\Tor_\BZ( M_2 \otimes \beta ) |.$$
Using Proposition \ref{mainx}(i), we get
 $$ |\Tor_\BZ(M_1\otimes \BZ[A_\Gamma]| \sim |\Tor_\BZ(M_2\otimes \BZ[A_\Gamma]|,$$
 which means $M_1 \sim M_2$.
\end{proof}
\subsection{Proof of Theorem \ref{T4}}
\begin{proof}  Since $M_1$ and $M_2$ are pseudo-isomorphic, there are pseudo-zero $N_1$ and $N_2$ such that
$$ 0 \to N_1 \to M_1 \to M_2 \to N_2 \to 0$$
is exact. Then we have the following exact sequences
\be 0 \to M_1/N_1 \to M_2 \to N_2 \to 0 \label{e81}
\ee
\be 0 \to N_1 \to M_1 \to M_1/N_1 \to 0 \label{e82}
\ee
From \eqref{e81} and Lemma \ref{mainn} we have $M_1/N_1 \sim M_2$, while  from \eqref{e82} and Lemma \ref{l908}  we have
$ M_1/N_1 \sim M_1$. It follows that $M_1 \sim M_2$,  which is equivalent to the statement of Theorem \ref{T4}.
\end{proof}

\def\TOR{{\operatorname{Tor}}}

\subsection{The case when $M$ is torsion-free} \begin{proposition} Suppose $M$ is a torsion-free finitely generated $\cR$-module. Then
$$ \lim_{\la \Gamma \ra \to \infty} \frac{\log | \Tor_\BZ (M\otimes \BZ[\BZ^n/\Gamma])|}{|\BZ^n /\Gamma|} =0.$$
\label{T5}
\end{proposition}

We first prove the following lemma.

\begin{lemma} Suppose $N, M_1$ and $M_2$ are finitely generated $\cR$-modules and

\be  0 \to M_1 \to M_2 \to N\to 0 \label{e5a}
\ee is exact.  If $M_2 \sim 0$,  then  $M_1 \sim 0$.
\label{mainn2}

\end{lemma}

\begin{proof}
  Using induction on the length  of a composition series \eqref{e71} of $N$  we can assume that $N =\cR/I$, where $I\subset \cR$ is a prime ideal.

  If $\cR/I$ is pseudo-zero, then by Lemma \ref{mainn}, $M_1 \sim M_2 \sim 0$.

  We will consider the remaining case, when $\cR/I$ is not pseudo-zero. Then $I$ is principal, $I=(p)$, where $p \in \cR$.

  If $p=0$, then $N=\cR$ is free, and the sequence \eqref{e5a} is split, $M_2 \cong M_1 \oplus \cR$. One clearly has $\Tor_\BZ(M_1\otimes \BZ[A_\Gamma] )= \Tor_\BZ(M_2\otimes \BZ[A_\Gamma] )$, and
  the statement follows.

 Suppose now $p\neq 0$.  Let $\beta = \beta(p;\Gamma)$. Tensoring  \eqref{e5a} with $\beta$,  the following is exact

$$ \dots \to \TOR^\cR_1(\cR/(p),\beta) \to M_1\otimes \beta  \to M_2 \otimes \beta  \to  \cR/(p )\otimes \beta \to  0.$$
By Lemma \ref{l5}(b),  the first term is 0. It follows  that
$M_1\otimes \beta $ is a subgroup of $M_2\otimes \beta $, and hence
$$ | \Tor_\BZ(M_1\otimes \beta )| \le  | \Tor_\BZ(M_2\otimes \beta )| .$$
By Proposition \ref{mainx}(i),
$$ | \Tor_\BZ(M_i\otimes \beta )| \sim | \Tor_\BZ(M_i\otimes \BZ[A_\Gamma] )|,$$
and since $ | \Tor_\BZ(M_2\otimes \BZ[A_\Gamma] )|\sim 1$, we can conclude that $ | \Tor_\BZ(M_1\otimes \BZ[A_\Gamma] )|\sim 1$, or $M_1 \sim 0$.
\end{proof}
\def\cF{\mathcal F}
\begin{proof}[Proof of Proposition \ref{T5}] Since $M$ is torsion free, the canonical map $ M \to V:= M \otimes_\cR \cF$, where $\cF$ is the fractional field of $\cR$, is an embedding. This means $M$ is a lattice of $V$ with respect to $\cR$, and hence there is a free $\cR$-module $F$ such that $M$  embeds into $F$, see \cite[Chepter 7]{Bourbaki}. One has an exact sequence of finitely-generated $\cR$-modules
$$  0 \to M  \to F \to N\to 0.$$
We have $F \sim 0$ since $F$ is a free $\cR$-module. From Lemma \ref{mainn2} we conclude that $M \sim 0$.
\end{proof}

\subsection{Proof of Theorem \ref{T3}}
\begin{proof}
 By Theorem \ref{ps}, $M$ is and $\Tor(M) \oplus M/\Tor(M)$ are pseudo-isomorphic. Hence by Theorem \ref{T4},
$$ M \sim \big ( \Tor(M) \oplus M/\Tor(M) \big)).$$
Since $M/\Tor(M)$ is torsion-free, by Theorem \ref{T5}, $M/\Tor(M) \sim 0$. Hence we have
\be
M \sim \Tor(M),
\ee
The proof is thus reduced to the case when $M$ is a torsion module, which had been proved by K. Schmidt, see \cite[Theorem 21.1]{Schmidt}.
\end{proof}

\subsection{Proof of Theorem \ref{T33}}

\def\tH {\tilde H}
\def\cD{\mathcal D}
\begin{proof}
Suppose $\cD$ is a chain complex of free finitely generated modules over a domain $R$,
$$ \dots \to R^{m_{i+1}} \overset{\partial_{i+1}} {\longrightarrow}  R^{m_{i}} \overset{\partial_{i}} {\longrightarrow}  \dots $$
  For our application either $R=\cR$ or $R=\BZ$.

  In the exact sequence
  $$0\to  (\ker \partial_i /\im \partial_{i+1}) \to (R^{m_i}/\im \partial_{i+1}) \to (R^{m_i}/\ker \partial_i) \to 0$$
 the first module is $H_i(\cD)$, the second $\coker \partial_{i+1}$. Since the third is a torsion free $R$-module, one has

 \be
 \Tor_R (H_i(\cD)) = \Tor_R (\coker \partial_{i+1}).
 \label{e0101}
 \ee

Suppose now $\cC$ is a chain complex of free finitely generated $\cR$-modules of the form

$$   \dots \to \cR^{m_{i+1}} \overset{\partial_{i+1}} {\longrightarrow}  \cR^{m_{i}} \overset{\partial_{i}} {\longrightarrow}  \dots $$

Apply \eqref{e0101} to the above chain complex, we have

\be  \Tor(H_i(\cC)) = \Tor(M),
\label{e0105}
\ee
where $M= \coker \partial_{i+1}$ which has a presentation

\be
 \cR^{m_{i+1}} \overset{\partial_{i+1}} {\longrightarrow}  \cR^{m_{i}} \to M \to 0.
 \label{e0106}
 \ee
Tensoring \eqref{e0106} with $\BZ[A]$,  where $A = \BZ^n/\Gamma$, we get the exact sequence
$$ \BZ[A] ^{m_{i+1}} \overset{\partial_{i+1,\Gamma}} {\longrightarrow}  \BZ[A] ^{m_{i}} \to M \otimes \BZ[A]\to 0,$$
from which it follows that

\be M \otimes \BZ[A] = \coker(\partial_{i+1,\Gamma}).
\label{e0103}
\ee

The complex $\cC\otimes \BZ[A]$ is

\be
  \dots \to \BZ[A] ^{m_{i+1}} \overset{\partial_{i+1,\Gamma}} {\longrightarrow}  \BZ[A] ^{m_{i}} \overset{\partial_{i,\Gamma}} {\longrightarrow}  \dots
  \label{e0102}
  \ee

Apply \eqref{e0101} to the chain complex \eqref{e0102}, considered as complex over $\BZ$, we get

$$  \Tor_\BZ(H_i(\cC \otimes \BZ[A]))  = \Tor_\BZ( \coker \partial_{i+1,\Gamma}).
$$
which, with \eqref{e0103}, gives

\be  \Tor_\BZ(H_i(\cC \otimes \BZ[A])) = \Tor_\BZ (M \otimes \BZ[A]).
\label{e0104}
\ee

Theorem \ref{T3}, with identity \eqref{e0104}, gives
$$ \limsup_{\la \Gamma \ra \to \infty} \frac{ \log |  \Tor_\BZ(H_i(\cC \otimes \BZ[A])) | }{|\BZ^n/\Gamma|}= \BM(\Delta(\Tor(M)),$$
From which and \eqref{e0105} we have
$$ \limsup_{\la \Gamma \ra \to \infty} \frac{ \log |  \Tor_\BZ(H_i(\cC \otimes \BZ[A])) | }{|\BZ^n/\Gamma|}= \BM(\Delta(\Tor(H_i(\cC))),$$
which completes the proof of Theorem \ref{T33}.\end{proof}

\section{Homology of abelian covering} \label{S4}

\subsection{Alexander polynomials of links}

\def\tY{\tilde Y }

Suppose $Z$ is an oriented integral homology 3-sphere, i.e. $H_i(Z,\BZ)\cong H_i(S^3,\BZ)$, and $L\subset Z$ is an oriented link with $n$ ordered components. Let $N(L)$ be a small open tubular neighborhood of $L$ and
$X= Z\setminus N(L)$. By Alexander duality
$H_1(X,\BZ)\cong \BZ^n$, and there is a natural identification of $H_1(X,\BZ)$ with $\BZ^n$ such that $t_i$ corresponds to the meridian of the the $i$-th component of the link. We fix such an identification of $H_1(X,\BZ)$ with $\BZ^n$.

Let $\tX $ be the abelian covering corresponding to the abelianization  $\pi_1(X) \to H_1(X,\BZ) = \BZ^n$. The homology groups $H_i(\tX,\BZ)$ has a structure of $\cR= \BZ[\BZ^n]$ module.
The {\em Alexander polynomials} $\Delta_i(L)$ (or $\Delta_i(L  \subset Z)$), by definition, are the polynomials $\Delta_i(H_1(\tX,\BZ))$. Recall that if $j$ is the smallest index such that
$\Delta_j(H_1(\tX,\BZ))\neq 0$, then one defines $\Delta(H_1(\tX,\BZ))= \Delta_j(H_1(\tX,\BZ))$. We also define $\Delta(L) = \Delta(H_1(\tX,\BZ))$.

Note that $X$ has Euler characteristic 0.
% Instead of $X$ one usually works with its 2-dimensional homotopically equivalent CW complex.
 It is known that $X$ is  homotopic to a finite 2-dimensional CW-complex $Y$, with 1 0-cell, $m+1$ 1-cell $a_1,\dots a_{m+1}$ and $m$ 2-cell $b_1, \dots, b_m$, for some number $m$. Certainly $m \ge n$.
Let $\rho: \pi_1(Y) \to H_1(Y,\BZ)$ be the
standard abelianization map. By choosing an appropriate CW-structure, we can assume further that  $\rho(a_i)=t_i$ for $i=1,\dots, n$.

Let $\tY$ be the abelian covering of $Y$ corresponding to the abelianization $\rho: \pi_1(Y) \to H_1(Y,\BZ) = \BZ^n$. The CW-complex of $\tY$ can be considered as a chain complex over $\cR= \BZ[t^{\pm 1}_1 , \dots, t^{\pm 1}_n]$, and has the form

\be
 0\to \cR^{m} \overset{\partial_{2}} {\longrightarrow}  \cR^{m+1} \overset{\partial_{1}} {\longrightarrow} \cR \to 0  .
 \label{e3030}
 \ee
Here

$$ \partial_1 = \begin{pmatrix} 1- \rho(a_1) \\ 1- \rho(a_2) \\ \dots \\ 1-\rho(a_{m+1}) \end{pmatrix}.$$
and $\partial_2$  is an $m \times (m+1)$-matrix with entries in $\cR$ which can be calculated using Fox derivative.
There is only one 0-cell of $\tY$, denoted by $O$. The lift of $a_i$ beginning at $O$ will be denoted by $\tilde a_i$, $i=1,\dots,m+1$.

\begin{remark} The module $M_2=\coker (\partial _2)$ is known as the Alexander module.  In some text eg \cite{Hillman}, the Alexander polynomials are defined as $\Delta_i(M_2)$, which differ from ours only by a shift of index:
 $\Delta_i(H_1(\tX,\BZ)) = \Delta_{i+1}(M_2)$ since both $H_1(\tX,\BZ)$ and $M_2$ have the same $\cR$-torsion, see \eqref{e0101}. In particular, $\Delta(L) = \Delta(H_1(\tX,\BZ)) = \Delta(M_2)$.
\end{remark}

\def\diag{\operatorname{diag}}
\subsection{Homology of the branched covering} Suppose $\Gamma \subset \BZ^n$ is a subgroup of finite index, and $A = A_\Gamma= \BZ^n/\Gamma$.
Let $X_\Gamma$ and $Y_\Gamma$ be the covering of $X$ and $Y$ respectively corresponding to the epimorphism $\pi_1 \to H_1 \to A$. Then the CW complex of $Y_\Gamma$ is
$\cC(\tY) \otimes_\cR \BZ[A]$:

\be
  0\to \RR ^{m} \overset{\partial_{2,\Gamma}} {\longrightarrow}  \RR^{m+1} \overset{\partial_{1,\Gamma}} {\longrightarrow} \RR \to 0
  \label{e401}
  \ee

The branched covering $X_\Gamma^\br$, by definition, is obtained from $X_\Gamma$ by Dehn fillings as follows. The boundary  of $X$ is the union of $n$ tori, each surrounding a
link component. The boundary of $X_\Gamma$ is also the union of several tori, each is the covering of one of tori in the boundary of $X$. Suppose $T$ is a torus in the boundary of $X_\Gamma$ covering
the $i$-th torus of the boundary of $X$. There is a simple closed curve $C$ on $T$ covering the meridian of the $i$-th torus. To every boundary component $T$ one does the Dehn filling on $T$ that kills the homology class of $C$. The resulting 3-manifold is $X_\Gamma^\br$.

The homology group $H_1(X_\Gamma^\br,\BZ)$ is the the quotient of $H_1(X_\Gamma,\BZ)$ by the relation $C=0$, for all the curves $C$ described in the above Dehn filling operation.
The difficulty with working $H_1(X_\Gamma^\br,\BZ)$ is these relations $C=0$ are local, they cannot be obtained from a global relation in terms of $\cR$-modules.

We now describe a universal $\cR$-module through which the $\BZ$-torsion of $H_1(X_\Gamma^\br,\BZ)$ can be calculated.

Recall that $\partial_2$ in \eqref{e3030} is an $m\times m+1$ matrix. Let $I_{n,m+1}$ be the $n\times (m+1)$ matrix obtained from the identity $n\times n$ matrix by adding $(m+1-n)$ columns of 0, and $T$ be the $n\times n$ diagonal matrix
$$T = \diag(1-t_1, \dots, 1-t_n).$$

\begin{proposition} Let $M$ be the $\cR$-module with the following presentation matrix
$$ \begin{pmatrix} \partial_2 & 0 \\ I_{n,m+1} & T
\end{pmatrix}, $$
which has  size $(m+n) \times (m+1+n)$.
Then for any subgroup $\Gamma\subset \BZ^n$ of finite index we have
$$ \Tor_\BZ \big( H_1(X_\Gamma^\br,\BZ)   \big) \cong \Tor_\BZ \big( M \otimes _\cR \BZ[A_\Gamma]   \big).$$
\label{p604}
\end{proposition}

\begin{proof}

% defined by the exact sequence
% $$ \cR^m \times \cR^n \overset{D} {\longrightarrow} \cR^{m+1} \times \cR^n \to M \to 0,$$
% where $D(x,y) = (\partial_2 x + y, T y)$.

For $i=1,\dots,n$ let $d_i=d_i(\Gamma)$ be the degree of $t_i$ in $A=\BZ^n/\Gamma$, and
$ u_i = \sum_{l=0}^{d_i-1} (t_i)^l \in \BZ[A]$.

The homology of the branched covering $H_1(X_\Gamma^\br,\BZ)$ is

$$ H_1(X_\Gamma^\br,\BZ)= H_1(Y_\Gamma,\BZ)/(Rel),$$
where $(Rel)$ is the $\RR$-submodule of $H_1(Y_\Gamma,\BZ)$ generated by $u_i\, \tilde a_i, i=1,\dots,n$.

By definition,  $H_1(Y_\Gamma,\BZ)$ is $H_1$ of the complex \eqref{e401}. By adding relations  $ u_i\, \tilde a_i=0, i=1,\dots,n$, we see that $H_1(X_\Gamma^\br,\BZ)$ is $H_1$ of the following complex
\be
 0\to \RR ^{m} \times \RR^n  \overset{D_1} {\longrightarrow}  \RR^{m+1} \overset{\partial_{1,\Gamma}} {\longrightarrow} \RR \to 0
 \label{e402}
 \ee
where $D_1(x,y) = \partial_{2,\Gamma} (x) + U'(y)$, with  $U': \RR^n \to \RR^{m+1}$ being  the $\RR$-linear map defined by
$$ U'(x_1,\dots,x_n) = \big( u_1\, x_1, \dots ,u_n \, x_n, 0, \dots, 0\big).$$

 Let $U: \RR^n \to \RR^n$ be the $\RR$-linear map given by the diagonal matrix $ U = \diag(u_1,\dots, u_n)$. Certainly $\ker U= \ker U'$.

Applying \eqref{e0101} to the chain complex \eqref{e402}, we get
\be \Tor_\BZ \big( H_1(X_\Gamma^\br,\BZ) \big) \cong \Tor_\BZ \coker D_1.
\label{e601}
\ee
The map $U':\RR^n \to \RR^{m+1}$ descends to $U'': (\RR^n/\ker U') \to \RR^{m+1}$, hence
$\coker D_1 = \coker D_1'$, where
$$ D_1' : \RR ^{m} \times (\RR^n /\ker U) \to \RR^{m+1},$$
defined by $D_1'(x,y) = \partial_2(x)+ U''(y)$. From \eqref{e601} we have
\be \Tor_\BZ \big( H_1(X_\Gamma^\br,\BZ) \big) \cong \Tor_\BZ \big( \coker D_1'\big).
\label{e6011}
\ee

% By tensoring with $\RR =\BZ[\BZ^n /\Gamma$, we get $T_\Gamma: \RR^n \to \RR^n$.

 By tensoring $T: \cR^n \to \cR^n$ with $\BZ[A]$, we get $T_\Gamma: \BZ[A]^n \to \BZ[A]^n$, which is given by a diagonal matrix.
Note that  $T_\Gamma U=0$, i.e. $T_\Gamma$ is 0 on the image of $U$, hence $T_\Gamma$ descend to a map $T'_\Gamma: (\RR^n/\im U)\to \RR^n$.

We have the following commutative diagram with exact vertical lines
$$
\begin{CD}
0  @ >>>   \RR^{m+1} \times (\RR^n/\ker U)   @ > D'_1>>         \RR^{m+1}  @ >>> 0 \\
@VVV   @VV i_1V  @Vi_2 VV @VVV \\
           0   @>>>  \RR^{m+1} \times \RR^n   @ > D_\Gamma>>  \RR^{m+1} \times \RR^n  @  >>> 0 \\
          @VVV   @VV j_1V  @Vj_2 VV @VVV \\
           0 @ >>>          ( \RR^n/\im U) @ > T_\Gamma' >>    \RR^n @ >>> 0
           \end{CD}
           $$
where $i_1(x,y) = (x, U(y))$, $i_2(x) = (x,0)$, $j_1(x,y) = (0, y)$, $j_2(x,y)= y$, and $D_\Gamma$ is the matrix of presentation of $M$, tensoring with $\RR$.

Let the first complex be $\cD_1$, the second $\cD_2$, and the 3-rd $\cD_3$. From the exact sequence $0\to  \cD_1 \to \cD_2 \to \cD_3\to 0$ we have a long exact sequence

\be
\dots H_1(\cD_3) \to H_0(\cD_1) \to H_0(\cD_2) \to H_0(\cD_3) \to 0.\label{e442}
\ee
The first term is $0$ and the last term is free abelian group, by Lemma \ref{l441} below. Hence the second term and the third term in \eqref{e442} have the same $\BZ$-torsion.
Since  $ H_0(\cD_1) =\coker D'_1$ and $H_0(\cD_2) = M \otimes \BZ[A]$, we have
$$\Tor_\BZ \big (\coker D'_1 \big) = \Tor_\BZ \big( M \otimes \BZ[A]\big) ,$$
which, together with \eqref{e6011}, proves the proposition.
\end{proof}

 \begin{lemma} For chain complex $\cD_3$
 $$ 0 \to (\RR^n/\im U) \overset{T_\Gamma' } {\longrightarrow}  \RR^n \to 0$$
 one has $H_1(\cD)=0$ and $H_0(\cD)$ is a free abelian group.
 \label{l441}
 \end{lemma}

\begin{proof} This is the same as to show that for the chain $\cD'$
$$ 0 \to \RR^n \overset{U} {\longrightarrow} \RR^n \overset{T_\Gamma } {\longrightarrow}  \RR^n \to 0$$
one has $H_1(\cD')=0$, and $H_0(\cD')$ is free abelian.

Since both $U$ and $T_\Gamma$ are diagonal, $\cD'= \bigoplus_{i=1}^n \cD_i'$, where $\cD'_i$ is the complex
$$ 0 \to \RR \overset{u_i} {\longrightarrow} \RR^n \overset{1-t_i } {\longrightarrow}  \RR \to 0.$$
As seen in subsection \ref{ss4}, the principal ideals $(u_i)$ and $(1-t_i)$ are  annihilator of each other, hence $H_1(\cD_i')=0$.
Besides, the ideal $(1-t_i)$ is  primitive as a lattice in $\BZ[A]$, hence $H_0(\cD_i')$ is a free abelian group.
\end{proof}

\subsection{Proof of Theorem \ref{T2}}
\begin{proof} Part (a), the case of non-branched covering, follows immediately from Theorem \ref{T7}.

Let us consider the case of branched covering. By Proposition \ref{p604}, we have

\begin{align} \limsup_{\la \Gamma \ra \to \infty}\frac{\log |\Tor_\BZ(H_1(X_\Gamma^\br, \BZ))|}{|\BZ^n /\Gamma|} & = \limsup_{\la \Gamma \ra \to \infty}\frac{\log \Tor_\BZ(M \otimes \BZ[A_\Gamma])}{|\BZ^n /\Gamma|} \notag \\
&= \BM(\Delta (M))  \quad \text{ by Theorem \ref{T3}}.
\label{e701}
\end{align}

The module  $M_1= \cR^n/T (\cR^n)$ has a free resolution
$$ 0 \to \cR^n  \overset{T} {\longrightarrow} \cR^n \to M_1 \to 0,$$
hence its projective dimension is $ 1$. Also $\Delta_0(M_1) = \prod_{i=1}^n (1-t_i)$.

Let  $M_2= \coker \partial_2$. From the matrix of presentation of $M$ we see that there is an exact sequence

$$ 0 \to M_2 \to M \to M_1 \to 0.$$

Since the projective dimension of $M_1$ is $\le 1$, by \cite[Theorem 3.12]{Hillman}, one has
\begin{align*}
 \Delta_j(M) & = \Delta_j(M_2)  \Delta_0(M_1)\\
 &= \Delta_j(M_2) \prod_{i=1}^n (1-t_i).
 \end{align*}

 It follows that $\Delta(M) = \Delta(M_2) \prod_{i=1}^n (1-t_i)$. Since $\BM(1-t_i)=1$ and $\Delta(M_2)=\Delta(L)$, we have
 $$ \BM(\Delta(M))= \BM(\Delta( L)),$$
 from which and \eqref{e701} one gets part (b) of Theorem \ref{T2}.
 \end{proof}

\section{converging sequences}
\label{last}

\subsection{Statement}

\def\dd{\mathbf r} For a non-zero vector $x \in \BR^n$ let $\dd(x)= x/||x|| \in S^{n-1}$ be the unit vector positively colinear with $x$. Here $S^{n-1}$ is the $(n-1)$-dimensional sphere
of unit vectors in $\BR^n$.
For a subgroup $\Gamma\subset \BZ^n$ of finite index let $d_i= d_i(\Gamma)$ be the degree of $t_i$ in the quotient group $A_\Gamma = \BZ^n /\Gamma$.
Let $\dd(\Gamma)= \dd(d_1,\dots,d_n)\in S_+^{n-1}$, the part of $S^{n-1}$ with non-negative coordinates.

\begin{Thm} Suppose $M$ is a finitely-generated $\cR$-module.
For any  $\kappa \in S^{n-1}_+$, there exists a sequence of finite index subgroups $\Gamma_s\subset \BZ^n, s=1,2,\dots $ such that
$$ \lim_{s\to \infty } \dd(\Gamma_s) = \kappa$$ and
$$ \limsup_{s \infty}
\frac{\log | \Tor_\BZ (M \otimes  \BZ[\BZ^n / \Gamma_s])
|}{|\BZ^n /\Gamma_s|}=  \BM (\Delta(M)). $$
\label{T8}
\end{Thm}

\begin{remark} One could prove a similar statement, replacing $M\otimes \BZ[\BZ^n/\Gamma_s]$ with $H_i(\cC \otimes \BZ[\BZ^n/\Gamma_s])$ like in Theorem \ref{T33}.
\end{remark}

The proof and methods of this section are
independent of Theorem \ref{T3}. It gives an alternative proof of ``half" of  Theorem \ref{T3}: The left hand side in the identity of Theorem \ref{T3} is greater than
or equal to the right hand side.
\subsection{A result of Bombieri and Zannier: reduction from $\BZ^n$ to $\BZ$} For $\bk \in \BZ^n$ let $\bk^\perp =\{ \bm \in \BZ^n \mid \bk\cdot \bm =0 \}$, where $\bk \cdot \bm$ is the usual dot product. Define
$$ \la \bk \ra = \la \bk^\perp \ra = \min \{ |x |, x \in \bk^\perp \setminus \{0\}\} .$$

The group homomorphism $\BZ^n \to \BZ$ given by $\bm \to \bm \cdot \bk$ gives rise to the algebra homomorphism
$\tau_\bk : \BQ[t^{\pm1}_1,\dots, t_n^{\pm 1}] \to \BQ[t^{\pm 1}]$ defined by
$$ \tau_\bk (t^\bm)  = t^{\bm \cdot \bk}.$$

The following is a deep result of Bombieri and Zannier \cite{Zannier,BMZ}, which was formulated as a conjecture by Schinzel.

\begin{thm}
Suppose $p_1,p_2\in \BQ[t^{\pm1}_1,\dots, t_n^{\pm 1}]$ are co-prime. There is a constant $C= C(p_1,p_2)$ such that if $\la \bk \ra > C$, then
$ \gcd(\tau_\bk(p_1) ,\tau_\bk(p_2))$ is the product of some (possibly none) cyclotomic polynomials.
\end{thm}

From this one can easily deduce the following.

\begin{proposition} Suppose $p_1,\dots, p_\ell\in \BZ[t^{\pm1}_1,\dots, t_n^{\pm 1}]$.
There exists a constant $C= C(p_1,\dots,p_k)$  such that if $\la \bk\ra > C$, then
$$ \gcd (\tau_\bk(p_1), \dots, \tau_\bk(p_\ell) ) = \phi\,  \tau_\bk (\gcd(p_1,\dots, p_\ell)),$$
where $\phi$ is a product of cyclotomic polynomials.
\label{bz}
\end{proposition}
\begin{proof} By dividing each of $p_i$ by $\gcd(p_1,\dots, p_\ell)$ we can assume that $\gcd(p_1,\dots, p_\ell)=1$.

We will use induction on $\ell$. But first make the following well-known observation on the coefficients of $\tau_\bk(p)$ of a polynomial  $p \in \BZ[t^{\pm 1}_1, \dots, t^{\pm 1}_n]$ having the form
$$ p = \sum_{\bm \in \mathcal N} c_\bm\, t^{\bm},$$
where $\mathcal N \subset \BZ$ is a finite set. Then
\be
 \tau_\bk(p) = \sum_{\bm \in \mathcal N} c_\bm\, t^{\bm\cdot \bk}.
 \label{e80001}
 \ee
If $\la \bk \ra $ is greater than $|\bm -\bm'|$ for any two $\bm \neq \bm' \in \mathcal N$, then $\bk\cdot \bm \neq \bk \cdot \bm'$, and
\eqref{e80001} shows that the coefficients of $\tau_\bk(p)$, in some order, are exactly the coefficients of $p$.

Now we proceed with induction. Suppose $\ell=2$. By Bombieri and Zannier result, over $\BQ[t^{\pm1}_1,\dots, t_n^{\pm 1}]$, $\gcd(\tau_\bk (p_1),\tau_\bk(p_2))=\phi$, a product of cyclotomic
polynomials. Hence over $\BZ[t^{\pm1}_1,\dots, t_n^{\pm 1}]$, $\gcd(p_1,p_2)=a \phi$, for some integer $a$. It follows that $a$ is the $\gcd$ of all the coefficients of $\tau_\bk (p_1)$ and $\tau_\bk(p_2)$. By the above observation, with $\la \bk \ra$ big enough, this means $a$ is the $\gcd$ of all the coefficients of $p_1$ and $p_2$, which must be 1. This proves the statement when $\ell =2$.

Now assume $ \ell \ge 3$.
One has
\begin{align*} \gcd (\tau_\bk(p_1), \dots, \tau_\bk(p_\ell) ) & = \gcd (\tau_\bk(p_1),   \gcd ( \tau_\bk(p_2), \dots, \tau_\bk(p_\ell) ) \\
&= \gcd (\tau_\bk(p_1), \phi\,  \tau_\bk \big(  \gcd ( p_2, \dots, p_\ell )\big) \quad \text{ by induction} \\
&=\phi '\,  \tau_\bk \Big(   \gcd \big(p_1,  \gcd ( p_2, \dots, p_\ell) \big)\Big)  \quad \text{ by case $\ell =2$} \\
&=\phi '\,  \tau_\bk \big(   \gcd (p_1,   p_2, \dots, p_\ell) \big).
\end{align*}
Here $\phi,\phi'$ are product of cyclotomic polynomials. The proof is completed.
\end{proof}
\subsection{A result of Lawton}Recall that the additive  Mahler measure $\BM(f)$ of $f \in \BC[t^{\pm 1}_1,\dots, t^{\pm 1}_n], f\neq 0 $, is defined by
$$ \BM(f) = \int_{\BS^n} \log |f(x)| d \sigma,$$
where $\BS^n$ is the $n$-dimensional torus, and $d\sigma$ is the invariant Haar measure on $\BS^n$ normalized so that $\int_{\BS^n}d \sigma =1$.

The Mahler measure is additive, $\BM(fg) = \BM(f) + \BM(g)$. It is known that
$\BM(f)=0$ is and only $f$ is a generalized cyclotomic polynomial, see e.g. \cite{Zannier,Schmidt}.

The following approximation result was a conjecture of D. Boyd, and was proved by Lawton, see e.g. \cite{Lawton,Zannier,Schmidt}.

\begin{thm}{\rm (Lawton)}
 Suppose $f \in \BC[t^{\pm 1}_1,\dots, t^{\pm 1}_n], f\neq 0 $. Then
$$ \lim_{\la \bk \ra \to \infty }\BM(\tau_\bk(f)) = \BM(f).$$
\end{thm}

\def\bks{\bk^{(s)}}
\def\ks{k^{(s)}}
\def\fQ{{\cR_1}}

\def\Ms{M^{(s)}}
\subsection{A converging sequence} The following follows from Bombieri-Zannier and Lawton results. Denote $\fQ= \BZ[t^{\pm1}]$.
 \begin{proposition} Let $\bks\in \BZ^n, s= 1,2,\dots$ be any sequence such that $\lim _{s\to \infty} \la \bks \ra =\infty$, and $M$ a finitely-generated $\cR$-module. Let $\Ms = M \otimes \fQ$, where
 $\fQ$ is considered as an $\cR$-module via $\tau_s:= \tau_{\bks}: \cR \to \fQ$. Then

$$ \lim _{s\to \infty} \BM( \Delta(\Ms)) = \BM(\Delta(M)).$$
\label{p5444}
\end{proposition}

\begin{remark} It not true in general that  $\lim _{s\to \infty}  \Delta(\Ms) = \Delta(M)$.
\end{remark}

\begin{proof} Suppose $M$  has a presentation matrix $\partial$ of size $m_1 \times m_0$, with entries in $\cR$. Then $\Ms$ has presentation matrix $\tau_s(\partial)$, with entries in $\fQ$.

Let $j$ be the smallest integer such that $\Delta_j(M) \neq 0$.
This means all the $(m_0-j')$-minor of $\partial$ is 0 if $j' < j$, and if the $(m_0-j)$-minors of $\partial$ are
$p_1,\dots, p_\ell\in \cR$, then
$$\Delta(M) = \Delta_j(M) = \gcd(p_1,\dots,p_\ell).$$

Note that every minor of $\tau_s(\partial)$ is obtained from the corresponding minor
by by applying $\tau_s$.  It follows that all the $(m_0-j')$-minor of $\tau_s(\partial)$ is 0 if $j' < j$, and the $m_0-j$ minors are $ \tau_s(p_1), \dots, \tau_s(p_\ell)$.

By Proposition \ref{bz}, for $s$ big enough,

$$ \gcd(\tau_s(f_1), \dots, \tau_s(f_r)) = \phi \, \tau_s \big( \gcd (f_1,\dots, f_r)\big) ,$$
where $\phi $ is a product of cyclotomic polynomials. This means

$$ \Delta(\Ms) = \phi \, \tau_s \big(  \Delta(M)\big) .$$
Using additivity of the Mahler measure and the fact that the Mahler measure of a cyclotomic polynomial is 0, we have

$$ \BM (\Delta(\Ms)) = \BM \left ( \tau_s \big(  \Delta(M)\big)\right) .$$

Since $\la \bks \ra \to \infty$ as $s\to \infty$, by Lawton theorem, we have

$$  \lim _{s\to \infty} \BM( \Delta(\Ms))=  \lim _{s\to \infty}  \BM \tau_s \big(  \Delta(M)\big) = \BM(\Delta(M)).$$
\end{proof}

\subsection{Theorem \ref{T3}, the case $n=1$}
In the previous section we approximate $\BZ^n$ by $\BZ^n/(\bk^\perp)$, which has rank 1. Now we want to approximate abelian rank 1 group by
finite cyclic group. Here we give a short, independent of previous sections, proof of the case $n=1$ of Theorem \ref{T3}.
\begin{proposition} \label{pT31}
Suppose $M$ is a finitely generated $\fQ$-module.
Then
$$ \lim_{\ell \to \infty} \frac{\Tor_\BZ \big( M \otimes \BZ[\BZ/\ell]  \big) }{\ell} = \BM(\Delta(\Tor M)).$$
\end{proposition}

\begin{proof} The reason the case $n=1$ is easy is  $\BZ[\BZ/\ell] = \fQ/(1-t^\ell)$, with $(1-t^\ell)$ a principal ideal.

For an $\fQ$-module $N$ and an element $a\in \fQ$ let $_aN$ be the $a$-torsion of $N$:
$$ _aN = \{ x \in N \mid ax =0\}.$$
A homological interpretation of $_aN$ is  ${\operatorname{Tor}}_1^\fQ(N, \fQ/(a)) = \ _aN$.
If $a|b$ then $_aN \subset\  _bN$.
If $N$ is a finitely-generated torsion module, then there is $b \in \fQ$, called a {\em universal annihilator} of $N$, such that for every $a \in \fQ$,
$$ _aN = \ _{\gcd(b,a)}N.$$
For example, such  $b$ can be the product of all the generators of all prime ideals associated to $N$. One can also define $b$ as the least common multiple of
 the family of annihilators of a generating family for $N$.

Since $M'=M/\Tor(M)$ is torsion free,  by \cite[Chapter VII]{Bourbaki},  there is a free $\fQ$-module $F$ such that $M'\subset F$ and $F/M'$ is a torsion module.
Let $f$ be a universal annihilator of $F/M'$. Decompose
$f= f_1 f_2$, where $f_1$ is the product of all cyclotomic factors (with multiplicity) in the prime decomposition  of $f$.
The identity map $F \to F$ descends to a surjective map
$$ F/(f_1F) \twoheadrightarrow \  _{f_1}(F/M').$$
Since $f_1$ is monic, $F/f_1F$ is a finitely generated $\BZ$-module. It follows that $ _{f_1}(F/M')$ is a finitely generated $\BZ$-module, hence its $\BZ$-torsion part is a finite set.

Tensoring the exact sequence
$$ 0 \to M' \to F \to F/M' \to 0$$
with $\BZ[\BZ/\ell]$, we get
$$ 0 \to {\operatorname{Tor}}_1^\fQ(F/M', \BZ[\BZ/\ell] ) \to M' \otimes \BZ[\BZ/\ell]   \to F \otimes\BZ[\BZ/\ell]  \to (F/M') \otimes\BZ[\BZ/\ell]  \to 0.$$
Since $F$ is a free $\fQ$-module, the third term is a free $\BZ$-module. It follows that the $\BZ$-torsions of the first and the second terms are the same
$$\Tor_\BZ \big( {\operatorname{Tor}}_1^\fQ(F/M', \BZ[\BZ/\ell] )\big)  = \Tor_\BZ (M' \otimes \BZ[\BZ/\ell]).$$
Note that
$$ {\operatorname{Tor}}_1^\fQ(F/M', \BZ[\BZ/\ell]) = \  _{(1-t^\ell)}(F/M')$$
 is a subset of $\ _{f_1}(F/M')$ since $1-t^\ell$ is a product of cyclotomic polynomial.
Since $ |\Tor_\BZ( \, _{f_1}(F/M'))| $ is finite  and does not depend on $\ell$, we conclude that $\Tor_\BZ (M' \otimes \BZ[\BZ/\ell])$ is finite and bounded from above.

 Tensoring the exact sequence
$$ 0 \to \Tor(M) \to M \to M' \to 0$$
with $\BZ[\BZ/\ell] = \fQ/(1-t^\ell)$, taking into account ${\operatorname{Tor}}_1^\fQ(M',\fQ/(1-t^\ell)=0$, we have
$$ 0  \to \Tor(M) \otimes \BZ[\BZ/\ell]  \to M \otimes \BZ[\BZ/\ell]  \to M' \otimes \BZ[\BZ/\ell] \to 0.$$
Since the $\BZ$-torsion of the last term is bounded, we see that the $\BZ$-torsion parts of the first two terms have the same growth, i.e. $M \sim \Tor(M)$.
The Proposition now follows from the case of torsion modules, which was known \cite{Schmidt} (see also \cite{GS,Riley,Luck_book}).
\end{proof}

\subsection{Converging sequences}
\begin{lemma} Let $M$ be a finitely generated $\cR$-module, and  $\bks\in \BZ^n, s= 1,2,\dots$ be any sequence such that $\lim _{s\to \infty} \la \bks \ra =\infty$ and $\gcd(\ks_1,\dots,\ks_n)=1$ for $s \ge 1$.
For each positive integer $j$ define the subgroup $\Gamma_{s,j} \subset \BZ^n$ by
$$\Gamma_{s,j}= (\bks)^\perp  + j\, \bks.$$
For every $s$ there exists an integer $\eta_s >0$ such that for every sequence $j_s > \eta_s$, we have

 $$ \lim_{s\to \infty } \frac{\log |\Tor_\BZ \left( M \otimes \BZ[A_{\Gamma_{s,j_s}}]    \right)  |}{|\BZ^n/\Gamma_{s,j_s} |} = \BM(\Delta(M)).$$
 \label{l7401}
\end{lemma}

\begin{proof}

% We will write $\tau_s$ for $\tau_{\bks}: \cR\to  \cR_1$, which makes $\cR_1$ an $\cR$-module.

 It is easy to see that for any  $\bk=(k_1,\dots,k_n) \in \BZ^n$ with $\gcd(k_1,\dots,k_n)=1$, and $0< j \in \BZ$,
  the map $\bm \to \bm \cdot \bk \pmod {j}$ is an isomorphism between $\BZ^n /(\bk^\perp \oplus j \bk  ) $ and $\BZ/(j|\bk|^2)$.

% $$ \tau_{s,j}: \BZ[\BZ^n] \to \BZ[\BZ/(j \, ||\bks||^2)].$$

It follows that, as $\cR$-modules,

$$ M \otimes_\cR \BZ[\BZ^n /\Gamma_{s,j}]
%\cong \left ( M \otimes \BZ[\BZ^n /(\bks)^\perp) \right) \otimes \BZ[\BZ^n /\Gamma_{s,j}]
\cong  \Ms \otimes_\fQ \BZ[\BZ/(j \ |\bks|^2)]
%% \cong \Ms \otimes_\fQ \fQ/(1-t^{j \|\bks||^2}
.$$

Since $|\BZ^n/\Gamma_{s,j} |= j \, |\bks|^2$, one has

\begin{align}
\lim_{j\to \infty} \frac{\log |\Tor_\BZ \left( M \otimes \BZ[A_{\Gamma_{s,j}}]    \right)  |}{|\BZ^n/\Gamma_{s,j} |} & = \lim_{j\to \infty} \frac{|\Tor_\BZ \left( \Ms \otimes_\fQ \BZ[\BZ/(j \|\bks||^2) ]    \right)  |}{j |\bks|^2},  \notag  \\
\lim_{j\to \infty} \frac{\log |\Tor_\BZ \left( M \otimes \BZ[A_{\Gamma_{s,j}}]    \right)  |}{|\BZ^n/\Gamma_{s,j} |} &=  \BM\big( \Delta(\Ms)   \big) \quad \text{by Proposition \ref{pT31}}.
\label{e802}
\end{align}

Let  $ a_{s,j}$ be the left hand side of \eqref{e802}. From \eqref{e802}, for fixed $s$, there is $\eta_s >0$  such that if $j > \eta_s$, then
$$ | a_{s, j} - \BM\big( \Delta(\Ms) | < 1/s.$$

It is clear if $j_s > \eta_s$, then
$$\lim_{s\to \infty } a_{s,j_s} = \lim_{s\to\infty } \BM\big( \Delta(\Ms)) = \BM\big( \Delta(M)),$$
 where the last identity follows from Proposition \ref{p5444}.
\end{proof}

\subsection{Proof of Theorem \ref{T8}}

Assume the sequence $\bks$ of Lemma \ref{l7401} satisfies $\ks_i \neq 0$ for $i=1,\dots, n$. If we choose $j_s$ divisible by the product $\ks_1 \dots \ks_n$,  then $d_i(\Gamma_{s,j_s})= |j_s/\ks_i|$, and
$$\dd(\Gamma_{s,j_s}) = \dd(1/|\ks_1|,\dots,1/|\ks_n|).$$
 Thus Theorem \ref{T8} follows from Lemma \ref{l7401} and the following.

\begin{lemma} Suppose $\kappa \in S^{n-1}_+$. There exists $\bk^{(s)} =(\ks_1,\dots,\ks_n)  \in \BZ^n$ such that

(i) $\ks_1,\dots,\ks_n >0$  and have greatest common divisor $1$.

(ii) $\lim_{s\to \infty }  \dd (1/\ks_1,\dots,1/\ks_n ) = \kappa$.

(iii) $\lim_{s\to \infty } \la \bks \ra = \infty$.

\end{lemma}
\begin{proof} Let $S^{n-1}_{++}$ be the subset of $S^{n-1}$ consisting of points with all positive coordinates. Let $\Inv: S^{n-1}_{++} \to S^{n-1}_{++}$ be the map defined by
$$\Inv(x_1,\dots,x_n)=\dd(1/x_1,\dots,1/x_n).$$
It is clear that $\Inv$ is an involution, and hence is a auto-diffeomorphism of $S^{n-1}_{++}$.

The set $Q$ of all points of the form $\dd(k_1,\dots,k_n)$, with $k_1,\dots,k_n$ positive and co-prime, is dense in
$S^{n-1}_+$. If $\cL$ is a finite collection of hyperplanes in $\BR^n$, then
$Q\setminus \cL$ is still dense in $S^{n-1}_+$.

For $s \ge 1$ let  $P_s$ be the set of all points in $\BZ^n$ having norm $\le s$, and
$$\cL_s = \bigcup_{s \in P_s} \bk^\perp.$$
By definition , if $\bk \not \in \cL_s$, then $\la \bk \ra  > s$.
\def\bxs{{\mathbf x}^{(s)}}

$Q\setminus \cL_s$ is dense in $S^{n-1}_+$, hence so is $\Inv(Q \setminus \cL_s)$.  This implies there is $\bxs\in Q \setminus \cL_s$ such that
\be
 || \Inv(\bks) - \kappa|| < 1/s.
 \label{elast}
 \ee
By definition, $\bxs= \dd(\bks)$ for some $\bks=(\ks_1,\dots, \ks_n)$ with positive and co-prime $\ks_i$. Since $\bks \not \in \cL_s$, we have
$\la \bks \ra > s$, which establishes Property (ii).
Inequality \eqref{elast} shows that Property (iii) also holds.
\end{proof}

\end{document}